\newif\ifarxiv
\newif\ifspringer
\arxivtrue

\ifspringer
\RequirePackage{amsthm}
\documentclass[pdflatex,sn-mathphys,Numbered]{sn-jnl}

\makeatletter 
\def\@acknow{}%
\long\def\EarlyAcknow#1 \par{%
\def\@acknow{\abstractfont\abstracthead*{Acknowledgments}
#1\par}}%

\def\printabstract{\ifx\@acknow\empty\else\@acknow\fi\par%
    \ifx\@abstract\empty\else\@abstract\fi\par}
\makeatother

\usepackage{manyfoot}
\fi

\ifarxiv
\documentclass{amsart}
\usepackage[utf8]{inputenc}
\usepackage[letterpaper,left=1in,top=1in,right=1in,bottom=1in]{geometry}
\usepackage[numbers,square]{natbib}
\fi

\usepackage{graphicx}
\usepackage[dvipsnames]{xcolor}
\usepackage{amsthm, amsmath, amssymb, amsopn}
\usepackage{hyperref}
\usepackage[nameinlink]{cleveref}
\usepackage{mathrsfs,comment}
\usepackage{thm-restate}
\usepackage[normalem]{ulem}

\declaretheorem{theorem}
\declaretheorem[numberlike=theorem]{corollary}
\newtheorem{proposition}[theorem]{Proposition}
\newtheorem{lemma}[theorem]{Lemma}
\newtheorem{question}[theorem]{Question}
\theoremstyle{definition}
\newtheorem{definition}[theorem]{Definition}
\theoremstyle{remark}
\newtheorem{example}[theorem]{Example}

\DeclareMathOperator\csp{CSP}
\DeclareMathOperator\typ{typ}

\DeclareMathOperator\Clo{Clo}
\DeclareMathOperator\Hom{Hom}

\newcommand\ignore[1]{}
\newcommand{\alg}[1]{\ensuremath{\mathbf{#1}}}  
\newcommand{\mbf}[1]{\ensuremath{\mathbf{#1}}}
\newcommand{\str}[1]{\ensuremath{\mathcal{#1}}}

\newcommand\counting[1]{\ensuremath{c_{#1}}}
\newcommand\Kpoly{\ensuremath{\mathcal K_{\textnormal{poly}}}}
\newcommand\Kpolyeff{\ensuremath{\mathcal K_{\textnormal{eff}}}}
\newcommand\Ksurj{\ensuremath{\mathcal K^{s}_{\textnormal{poly}}}}
\newcommand\Ksurjeff{\ensuremath{\mathcal K^{s}_{\textnormal{eff}}}}
\newcommand\AddC[1]{#1^{\!*}}
\newcommand{\trac}[1][\alg A]{\mathcal{Z}(#1)}

\newcommand{\tuple}[1]{\mathbf{#1}}

\newcommand\abstractcontent{
    We provide an internal characterization of those finite algebras (i.e., algebraic structures) $\alg A$ such that the number of homomorphisms from any finite algebra $\alg X$ to $\alg A$ is bounded from above by a polynomial in the size of $\alg X$. 
    Namely, an algebra $\alg A$ has this property if, and only if, no subalgebra of $\alg A$ has a nontrivial strongly abelian congruence. 
    We also show that the property can be decided in polynomial time for algebras in finite signatures.
    Moreover, if $\alg A$ is such an algebra, the set of all homomorphisms from $\alg X$ to $\alg A$ can be computed in polynomial time given $\alg X$ as input.
   As an application of our results to the field of computational complexity, we characterize inherently tractable constraint satisfaction problems over fixed finite structures, i.e., those that are tractable and remain tractable after expanding the fixed structure by arbitrary relations or functions.}

\newcommand\acknowcontent{
\vskip-0.5cm
\begin{center}
\includegraphics[scale=0.2]{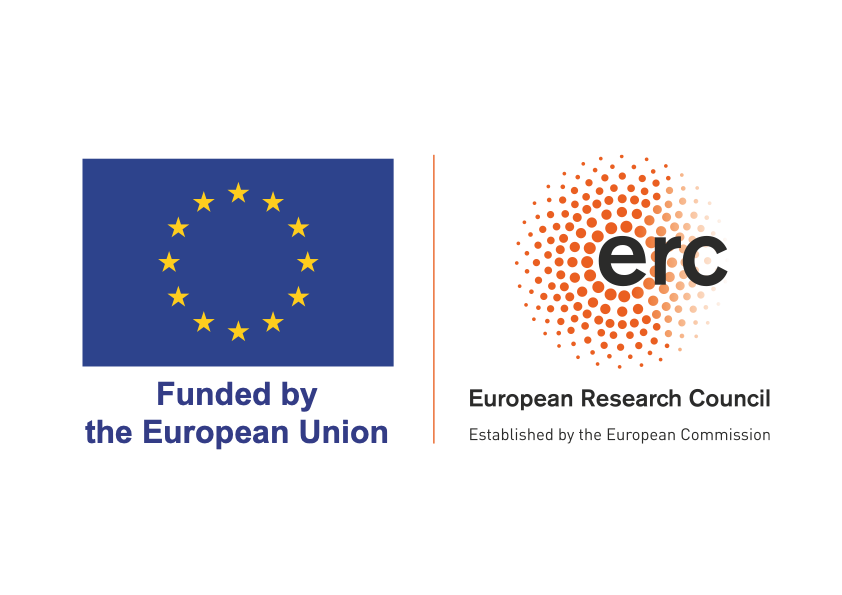} 
\end{center}
\vskip-0.5cm
Both authors have been supported by the European Research Council under the European Union's Horizon 2020 research and innovation programme (grant agreement 771005).
Libor Barto was also funded by the European Union (ERC, POCOCOP, 101071674). 
Views and opinions expressed are however those of the authors only and do not necessarily reflect those of the European Union or the European Research Council Executive Agency. Neither the European Union nor the granting authority can be held responsible for them.
}

\ifarxiv
\author{Libor Barto}
	\address{Department of Algebra, MFF UK, Sokolovsk\'a 83, 186 00 Praha 8, Czech Republic}
	\email{libor.barto@gmail.com}
    \urladdr{http://www.karlin.mff.cuni.cz/~barto/}
\author{Antoine Mottet}
	\address{Hamburg University of Technology, Research Group for Theoretical Computer Science, Hamburg, Germany}
	\email{antoine.mottet@tuhh.de}
	\urladdr{https://amottet.github.io}

\thanks{\acknowcontent}
\fi

\ifspringer

\author*[1]{\fnm{Libor} \sur{Barto}}\email{libor.barto@mff.cuni.cz}

\author*[2]{\fnm{Antoine} \sur{Mottet}}\email{antoine.mottet@tuhh.de}

\affil*[1]{\orgdiv{Department of Algebra, Faculty of Mathematics and Physics}, \orgname{Charles University}, \orgaddress{\street{Sokolovsk\'a 83}, \city{Prague}, \postcode{18675}, \country{Czech Republic}}
https://orcid.org/0000-0002-8481-6458}

\affil*[2]{\orgdiv{Research Group for Theoretical Computer Science}, \orgname{Hamburg University of Technology}, \orgaddress{\street{Blohmstr. 15}, \city{Hamburg}, \postcode{21079}, \country{Germany}}
https://orcid.org/0000-0002-3517-1745}

\EarlyAcknow{\acknowcontent}

\abstract{\abstractcontent}

\fi

\title{Finite Algebras with Hom-Sets of Polynomial Size}
\date{}

\ifspringer

\fi

\begin{document}

\maketitle

\ifarxiv
\begin{abstract}
    \abstractcontent
\end{abstract}
\fi

%
%

\section{Introduction}


The number of mappings from an $n$-element set $X$ to a fixed finite set $A$ is $|A|^n$, which is exponential in $n$ whenever $|A| \geq 2$. On the other hand, the number of homomorphisms from an $n$-element group $\alg X$ to a fixed finite group $\alg A$ is upper bounded by a polynomial in $n$, because groups have generating sets of logarithmic size and homomorphisms are determined by images of generators. The same polynomial upper bound in fact holds for any $n$-element algebra $\alg X$ in the signature of groups, as we will observe.

Our main result is an internal characterization of those finite algebras $\alg A$ for which the cardinality of $\Hom(\alg X,\alg A)$, the set of all homomorphisms from $\alg X$ to $\alg A$, is upper bounded by a polynomial in the cardinality of the universe of $\alg X$.  
It turns out that finite algebras typically have this property; those that do not must locally resemble an algebra with only unary operations. However, the reason for $\Hom(\alg X,\alg A)$ to have polynomial size is often different than $\alg X$ having a small generating set as in the case of groups. 

In this paper, by an algebra we mean a structure in a purely functional signature. In detail, a \emph{functional signature} $S$ consists of function symbols, each with associated nonnegative integer called \emph{arity}. An \emph{algebra} $\alg A$ in signature $S$ consists of a set $A$ called the \emph{universe} of $\alg A$ and interpretations of function symbols: for each $f \in S$ of arity $k$, the interpretation of $f$ in $\alg A$ is a $k$-ary \emph{operation} $f^{\alg A}$ on $A$, that is, a mapping $f^{\alg A} \colon A^k \to A$. An algebra is \emph{finite} if its universe is. 
Given two algebras $\alg X$ and $\alg A$ in the same signature, a homomorphism from $\alg X$ to $\alg A$ is a mapping $h\colon X\to A$ which \emph{preserves} every $f$ in the signature, that is,  $h(f^{\alg X}(x_1,\dots,x_k)) = f^{\alg A}(h(x_1),\dots,h(x_k))$ holds for all $x_1, x_2, \dots, x_k \in X$, where $k$ is the arity of $f$.

For a finite algebra $\alg A$, we denote by $\counting{\alg A}(n)$ the maximum value of $\left| \Hom(\alg X,\alg A)\right|$, where $\alg X$ ranges over all algebras in the signature of $\alg A$ with at most $n$ elements (i.e., $|X| \leq n$). The main result can be stated as follows.

\begin{theorem} \label{thm:main-official}
    The following are equivalent for a finite algebra $\alg A$.
    \begin{enumerate}
        \item\label{itm:main-official-counting-poly} $\counting{\alg A}(n) \in O(n^k)$ for some integer $k$. 
        \item \label{itm:main-official-strongly-abelian} No subalgebra of $\alg A$ has a nontrivial strongly abelian congruence.
    \end{enumerate}
\end{theorem}

Abelian congruences are among the basic concepts of \emph{commutator theory}~\cite{Commutator-theory} which emerged in the 1970s and provided useful generalizations of  concepts such as commutator, solvability, or nilpotence from groups to general algebras. Strong abelianness is a substantially stronger form of abelianness, e.g., no group has a nontrivial strongly abelian congruence. The concept orginated in McKenzie's investigation of representing finite lattices as congruence lattices of finite algebras~\cite{forbidden-lattices} and was further developed and applied e.g. in \cite{Hobby:1988,finite-complexity,Kearnes:2013}. 
Crucial for us is the connection of strong abelianness to  \emph{tame congruence theory}, a structure theory of finite algebras initiated in~\cite{Hobby:1988}. This theory is an essential tool for our proof.

We introduce strongly abelian congruences in~\Cref{subsec:strongly-abelian}, where we also show that the property in item~(\ref{itm:main-official-strongly-abelian}) of~\Cref{thm:main-official} is decidable in polynomial time for algebras in finite signatures.
The negative part of~\Cref{thm:main-official} is to provide a superpolynomial lower bound on $\counting{\alg A}$ in case that some subalgebra of $\alg A$ has a nontrivial strongly abelian congruence. The proof presented in \Cref{sec:nonmember} gives a nearly exponential lower bound $2^{\Omega(n^{1/k})}$ for $k = \lfloor \log_2 |A| \rfloor$ and we show that the bound is essentially optimal up to a constant in the exponent.
\Cref{sec:member} contains the proof of the more involved, positive part of~\Cref{thm:main-official}, where we employ the tame congruence theory. The polynomial upper bound is effective in that if the condition in item (\ref{itm:main-official-strongly-abelian}) is met and $\alg A$ is in finite signature, then homomorphisms from $\alg X$ to $\alg A$ can be algorithmically enumerated in polynomial time. We state these refinements in \Cref{sec:summary} and discuss further research directions.

The paper is meant to be understandable by every mathematician. It is largely self-contained except from some very basic results in universal algebra (reviewed in \Cref{subsec:very_basics}) and results from tame congruence theory (\Cref{thm:snags}, \Cref{subsec:tct}).

\subsection{Constraint satisfaction problems over fixed structures} 

Our characterization of finite algebras with polynomially-sized hom-sets is originally motivated by the investigation of the complexity of constraint satisfaction problems over fixed finite structures. 

In this discussion,  a signature may contain, apart from function symbols, also relation symbols with associated arities. A \emph{structure} $\str{A}$ consists of a universe and interpretations of function symbols as above, and additionally interpretations of relation symbols: for each relation symbol $R$ of arity $k$, its interpretation in $\alg A$ is a $k$-ary relation $R^{\str{A}} \subseteq A^k$. Given two structures $\str{X}$ and $\str{A}$ in the same signature, a homomorphism $h \colon \str{X} \to \str{A}$ is a mapping $h \colon X \to A$ that preserves all function symbols as well as every relation symbol $R$, i.e., $(h(x_1),\dots,h(x_k)) \in R^{\str A}$ whenever $(x_1, \dots, x_k) \in R^{\str X}$, where $k$ is the arity of $R$. A structure $\str{A}$ is a \emph{reduct} of a structure $\str{B}$ if they have the same universe $A=B$, the signature of $\str{A}$ is a subset of the signature of $\str{B}$ and symbols are interpreted in $\str{A}$ in the same way as in $\str{B}$. We also say that $\str{B}$ is an \emph{expansion} of $\str{A}$.
 
The \emph{constraint satisfaction problem (CSP) over $\str{A}$}, written $\csp (\str{A})$,
is the computational problem to decide whether a given input finite structure $\str{X}$ of the same signature as $\str{A}$ admits a homomorphism to $\str{A}$. The most investigated special case is when $\str{A}$ is a finite structure of finite, purely relational signature. One of the main goals \cite{Schaefer:1978,Feder:1999}, to obtain a characterization of the computational complexity of such CSPs, motivated much of  developments in universal algebra in the last 25 years. These efforts culminated in a celebrated dichotomy theorem obtained independently in~\cite{Bulatov:2017} and~\cite{Zhuk:2017,Zhuk:2020} that provides a complete complexity classification of these CSPs assuming $P \neq NP$: for every finite structure $\str A$ of finite purely relational signature, the problem $\csp(\str A)$ is solvable in polynomial time whenever a specific condition on $\str A$ is satisfied, and otherwise $\csp(\str A)$ is $NP$-complete (informally, it is ``a hardest'' problem solvable in nondeterministic polynomial time). We refer the reader to \cite{barto:2015c} for a short introduction to this area and to \cite{Barto:2017} for a more comprehensive survey. 

The CSP over general finite structures, not necessarily purely relational, was investigated in \cite{Feder:2004}. It is e.g.\ shown that, interestingly, even the purely algebraic setting extends the relational: for every finite relational structure $\str A$, there exists a finite algebra $\alg B$ (both of finite signature) such that $\csp(\str A)$ is equivalent to $\csp(\alg B)$ modulo polynomial-time Turing reductions. The CSP over general finite structures, in particular algebras, is thus one of the natural classes of computational problems to be systematically investigated after the dichotomy theorem for relational structures. Our paper with DeMeo \cite{DBLP:conf/lics/BartoDM21} provided some general results and partial classifications in this direction. In this work we have encountered an interesting phenomenon that the CSP over an algebra $\alg A$ is often solvable in polynomial time for a particularly strong reason: there are only polynomially many homomorphisms from the input algebra to $\alg A$ and they can be even enumerated in polynomial time. This development has driven us to investigating the class of algebras with polynomially bounded $c_{\alg A}$ and eventually led to the characterization in~\Cref{thm:main-official}. Notice that for a relational structure $\str A$ it is never the case that $\counting{\str A}$ is polynomially bounded unless $|A|\leq 1$, as witnessed by structures $\str X$ with symbols interpreted as empty relations. Even counting the number of homomorphisms to $\str A$ is typically a hard computational problem \cite{Bulatov-counting}.

If the homomorphisms from $\alg X$ to a fixed $\alg{A}$ can be enumerated in polynomial time, then $\csp(\alg A)$ is clearly solvable in polynomial time and so is $\csp(\str B)$ for any expansion $\str B$ of $\alg A$. In this sense, such algebras $\alg A$ have \emph{inherently tractable} CSPs. A consequence of the proof of~\Cref{thm:main-official} is that having polynomially bounded $\counting{\alg A}$ in fact characterizes inherently tractable CSPs. In the statement, the \emph{algebraic reduct} of $\str A$ is the reduct obtained by taking all the algebraic symbols in the signature of $\str A$.

\begin{restatable}{theorem}{cspcorollary} \label{thm:csp-corollary}
    Assuming $P \neq NP$, the following are equivalent for a finite structure $\str A$ of finite signature.
   \begin{enumerate}
       \item For every finite-signature expansion $\str B$ of $\str A$, the problem $\csp(\str B)$ is solvable in polynomial time.
       \item For no finite-signature expansion $\str B$ of $\str A$, the problem $\csp(\str B)$ is NP-complete. 
       \item $\counting{\alg A}(n) \in O(n^k)$ for some integer $k$, where $\alg A$ is the algebraic reduct of $\str A$.
   \end{enumerate}
\end{restatable}

A full  classification of the complexity of CSPs over finite algebras remains an open problem.

\subsection{Related numerical invariants} \label{subsec:related-counting}

\Cref{thm:main-official} relates a polynomial upper bound on the sequence $c_{\alg A}$ to an internal property of $\alg A$ for general finite algebras. 
We now mention a few similar results for other natural counting sequences associated to finite algebras.
In the following, $\mathcal{V}(\alg A)$ denotes the class of those algebras in the signature of $\alg A$ that satisfy all the identities satisfied in $\alg A$, see \Cref{subsec:very_basics}. We remark that a reflection of an  algebra $\alg X$ to $\mathcal{V}(\alg A)$ given in \Cref{prop:eq} is a useful step in the proof of the main theorem.

A result quite closely related to ours concerns the \emph{free spectrum} sequence, where the $n$th element is the cardinality of the $n$-generated free algebra in $\mathcal{V}(\alg A)$; equivalently, the number of $n$-ary term operations of $\alg A$. It is proved in~\cite{finite-complexity} that, for a finite $\alg A$ of finite signature, its free spectrum is bounded by a polynomial if, and only if, $\alg A$ is strongly nilpotent. 
As we explain in \Cref{subsec:negative_result}, a polynomial upper bound on the free spectrum gives us automatically a superpolynomial lower bound on $c_{\alg A}$. In fact, the negative part of \Cref{thm:main-official} is based on a simple and seemingly novel modification of free algebras. 

The \emph{G-spectrum} sequence \cite{generative-complexity} counts the number of at most $n$-generated algebras in $\mathcal{V}(\alg A)$ up to isomorphism; equivalently, the number of isomorphic types of homomorphic images of the $n$-generated free algebra in $\mathcal{V}(\alg A)$. The effort  to characterize finite algebras with polynomially bounded G-spectrum culminated in \cite{Poly-G-Spectrum}, which proves that all such algebras can be constructed in a specific way from  modules over finite rings of finite representation type and so-called \emph{matrix powers} of permutation groups. In \Cref{subsec:matrix-power}, matrix powers will provide examples of algebras with superpolynomial but subexponential growth of $c_{\alg{A}}$. Another intersection of \cite{Poly-G-Spectrum} with this paper is the crucial role of commutator theory and  tame congruence theory.

Finally, another numerical measure of the complexity of an algebra $\alg A$ comes from the investigation of CSPs over relational structures. It is obtained by counting for $n\geq 0$ the largest size of a minimal generating set for subalgebras of the $n$th power of $\alg A$.
The class of algebras for which this sequence is polynomially bounded was characterized in~\cite{Berman:2010} by means of the existence of term operations of $\alg A$ satisfying specific identities.
There is no obvious relation between small generating sets of subalgebras of powers of $\alg A$ and small generating sets of algebras $\alg X \in \mathcal{V}(\alg A)$; nevertheless, it follows from \Cref{thm:main-official} that $c_{\alg{A}} \in O(n^k)$ whenever $\alg A$ has the former property.

\section{Basics} \label{sec:basics}

We start this section by reviewing the very basic universal algebraic concepts and results in~\Cref{subsec:very_basics}.  This part can be skimmed over by readers acquainted with the subject; introductory books include \cite{BS,Bergman:2012} and a comprehensive resource is the book-series \cite{ALVIN-I,ALVIN-II,ALVIN-III}. 
In \Cref{subsec:poly-classes} we introduce and discuss the class of algebras admitting polynomially-many homomorphisms and their surjective and efficient variants. \Cref{subsec:general-facts}  provides several simple, but useful observations, which are then applied in~\Cref{subsec:examples} to show that homomorphisms to groups and semilattices can be efficiently computed. Finally, in~\Cref{subsec:strongly-abelian} we introduce strongly abelian congruences and show that their existence is efficiently decidable.

\subsection{Terms and basic constructions} \label{subsec:very_basics}
For the entire subsection we fix a purely functional signature $S$.
A \emph{term} over a set of variables $\{x_1 \dots, x_k\}$ is a formal meaningful expression that involves variables $x_1$, \dots, $x_k$ and function symbols in $S$. Given an algebra $\alg A$, every term, say $s$, over $\{x_1, \dots, x_k\}$ has a natural interpretation in $\alg A$ as a $k$-ary operation on $A$ which we denote by $s^{\alg A}$. Each such an operation is called a \emph{term operation} of $\alg A$. Note that each $f^{\alg A}$, $f \in S$ is also a term operation, these are called \emph{basic operations}. 
Also notice that homomorphisms not only preserve $f \in S$ but also all terms.

The set of all term operations of $\alg A$ is denoted $\Clo(\alg A)$ and the set of $n$-ary term operation of $\alg A$ is denoted $\Clo_n(\alg A)$.  Two algebras $\alg A$ and $\alg B$ with the same universes $A = B$ are \emph{term-equivalent} if $\Clo(\alg A) = \Clo(\alg B)$.
An algebra $\alg A$ is a \emph{term-reduct} of an algebra $\alg B$ if $\Clo(\alg A) \subseteq \Clo(\alg B)$. 

An \emph{identity} is a pair $(s,t)$ of terms over the same set of variables, also written $s\approx t$. An algebra $\alg A$ \emph{satisfies} an identity $s \approx t$ if  $s^{\alg A} =t^{\alg A}$. A \emph{variety} is the class of all algebras satisfying every identity from a fixed set $\Sigma$. For instance, the class of groups is the variety  determined in this way by the following $\Sigma_{grp}$ (here the signature $S$ consists consists of nullary symbol $1$, unary symbol $\ ^{-1}$ and binary symbol $\cdot$):
\begin{equation*}
\Sigma_{grp} = \{1 \cdot x \approx 1, \ x \cdot 1 \approx x, \ x \cdot x^{-1} \approx 1, \ x^{-1} \cdot x \approx 1, \ (x \cdot y) \cdot z \approx x \cdot (y \cdot z)\}.\tag{$\star$}\label{eq:Sigma_grp}
\end{equation*}
For an algebra $\alg A$, the variety \emph{generated} by $\alg A$ is the variety  determined by $\Sigma$ consisting of all identities satisfied by $\alg A$. 
By Birkhoff's HSP theorem~\cite{BirkhoffHSP}, varieties are exactly the classes closed under forming isomorphic copies, subalgebras, products, and quotients (or homomorphic images).  We now review the three constructions.

 An algebra $\alg B$ is a \emph{subalgebra} of $\alg A$, written $\alg B \leq \alg A$, if  $B\subseteq A$ and $f^{\alg B}=f^{\alg A}|_{B^k}$ for every symbol $f\in S$ (where $k$ is its arity). In particular, the definition implies that the universe of $\alg B$ must be closed under every operation of $\alg A$ -- it must be a \emph{subuniverse}. The smallest subuniverse (or subalgebra) of $\alg A$ containing a set $X \subseteq A$ is called the subuniverse (subalgebra) of $\alg A$ \emph{generated} by $X$. It is equal to the one-step closure of $X$ under term operations of $\alg A$ and, if $S$ is finite, it can be computed from $\alg A$  and $X$ in polynomial time by iteratively closing $X$ under the basic operations.

 The \emph{product} of algebras $\alg A_i$, $i \in I$ has as its universe the Cartesian product of the sets $A_i$ and the basic  operations are defined coordinate-wise. Of particular importance for us is the case when $\alg A_i = \alg A$. The product is then called the $I$th power and is denoted $\alg A^I$. We identify its universe with $A^I$, the set of all mappings from $I$ to $A$.

Given an equivalence relation $\alpha$ on $A$, we write $a/\alpha$ to denote the $\alpha$-equivalence class of $a$, that is, $a/\alpha=\{b\in A \mid (a,b)\in\alpha\}$. The set of equivalence classes is denoted  $A/\alpha = \{a/\alpha \mid a \in A\}$. 
We say that $\alpha$ is a \emph{congruence} of an algebra $\alg A$ if it is preserved by all (term or basic) operations of $\alg A$. It then makes sense to define the quotient algebra $\alg A/\alpha$ with universe $A/\alpha$ by $f^{\alg A/\alpha}(a_1/\alpha,\dots,a_k/\alpha) = f^{\alg A}(a_1,\dots,a_k)/\alpha$ for a $k$-ary $f \in S$.
For every algebra $\alg A$, the equality relation, denoted by $0_A$, is a congruence of $\alg A$, and so is the full relation $A\times A$, denoted $1_A$. If $\alg A$ has only those congruence, then it is called \emph{simple}. A congruence $\alpha$ of $\alg A$ is \emph{minimal} if it is not equal to $0_A$ and there is no congruence $\beta$ of $\alg A$ such that $0_A \subsetneq \beta \subsetneq \alpha$.
The smallest congruence of $\alg A$ containing a given $R \subseteq A^2$ is called the congruence \emph{generated} by $R$. It can be obtained by first forming the reflexive and symmetric closure, then closing under term operations, and finally forming the transitive closure (this takes polynomial time if the signature $S$ is finite).

The \emph{kernel} of a mapping $h: A \to B$ is the equivalence relation on $A$ defined by
$$
\ker(h) =\{(a,a')\in A^2 \mid h(a)=h(a')\}.
$$
The kernel of a homomorphism $h\colon\alg A \to \alg B$ is a congruence of $\alg A$. The quotient $\alg A/{\ker(h)}$ is isomorphic to the image of $h$ (i.e., the subalgebra of $\alg B$ with universe $h(A)$). 
Congruences of a quotient algebra $\alg A/\alpha$ bijectively correspond to congruences of $\alg A$ containing $\alpha$. Given a congruence $\beta \supseteq \alpha$, we denote by $\beta/\alpha$ the corresponding congruence of $\alg A/\alpha$.

\subsection{Polynomial classes and efficient variants} \label{subsec:poly-classes}

Recall that $\counting{\alg A}(n)$ is the sequence counting the maximum number of homomorphisms from an at most $n$-element algebra $\alg X$ (in the same signature as $\alg A$) to $\alg A$. We similarly define $\counting{\alg A}^s(n)$ counting the maximum number of surjective homomorphisms.
We define
\[
\Kpoly = \{\alg A \mid  \counting{\alg A}(n) \in O(n^k) \mbox{ for some $k$}\}
\]
and
\[
\Ksurj = \{\alg A \mid  \counting{\alg A}^s(n) \in O(n^k) \mbox{ for some $k$}\}
\]

The following observation explains the relationship between $\Kpoly$ and $\Ksurj$.

\begin{proposition}\label{prop:equivalences-Kpoly}
Let $\alg A$ be a finite algebra.
The following are equivalent.
\begin{enumerate}
    \item $\alg A$ is in $\Kpoly$.
    \item Every subalgebra $\alg B$ of $\alg A$ (including $\alg A$) is in $\Ksurj$.
\end{enumerate}
\end{proposition}
\begin{proof}
    (1) implies (2) follows from the fact that every surjective homomorphism $\alg X\to\alg B$ is in particular a homomorphism $\alg X\to\alg A$, giving $\counting{\alg B}^s\leq\counting{\alg A}$.
    
    For (2) implies (1), observe that every homomorphism $\alg X\to\alg A$ is a surjective homomorphism onto a subalgebra $\alg B$ of $\alg A$.
    In particular, $\counting{\alg A}\leq \sum_{\alg B\leq\alg A} \counting{\alg B}^s$, which is polynomially bounded by assumption.
\end{proof}

For an algebra $\alg A$ in a finite signature, a stronger requirement than $\alg A \in \Kpoly$ is that the homomorphisms $\alg X \to \alg A$ can be \emph{efficiently enumerated}: namely, that there exists a polynomial-time algorithm which, given an input algebra $\alg X$ in the same signature as $\alg A$, enumerates all the homomorphisms from $\alg X \to \alg A$. Note that the concrete representation of the input algebra $\alg X$ is irrelevant --- any reasonable encoding will do, e.g., for each symbol $f \in S$, the operation $f^{\alg X}$ can be represented by a natural encoding of its table.  

We extend the concept of efficient enumerability to algebras with infinite signatures by defining the class $\Kpolyeff$ as follows.
\begin{align*}
\Kpolyeff = \{\alg A \mid \ &\exists \alg B 
\mbox{ finite-signature reduct of $\alg A$ such that} \\
& \mbox{homomorphisms $\alg X \to \alg B$ can be efficiently enumerated} \} 
\end{align*}
The efficient variant $\Ksurjeff$ of $\Ksurj$ is introduced analogously. 
Observe that the efficient variant of~\Cref{prop:equivalences-Kpoly} holds as well and that we have the trivial inclusions
$$
\Kpolyeff \subseteq \Kpoly \subseteq \Ksurj \supseteq \Ksurjeff \supseteq \Kpolyeff.
$$

\subsection{General facts} \label{subsec:general-facts}

Most reasonable properties algebras only depend on their term operations rather than the concrete choice of basic operations. The following observation shows that membership in the polynomial classes is among such properties.

\begin{proposition}\label{prop:clone-Kpoly}
    Let $\alg A, \alg B$ be finite algebras such that $\Clo(\alg A) \subseteq \Clo(\alg B)$.
    If $\alg A$ is in one of the classes $\Kpoly,\Ksurj,\Kpolyeff,\Ksurjeff$, then so is $\alg B$. 
\end{proposition}
\begin{proof}
  Consider first the classes $\Kpoly,\Ksurj$. 
  For an algebra $\alg{Y}$ in the signature of $\alg{B}$, we define an algebra $\alg{X}$ in the signature of $\alg{A}$ with the same universe as $\alg{Y}$ as follows. 
  For every symbol $f$ in the signature of $\alg A$, there exists a term $t_f$ in the signature of $\alg B$ such that $f^{\alg A}=t_f^{\alg B}$ (as $\Clo(\alg A) \subseteq \Clo(\alg B)$). We set $f^{\alg X} = t_f^{\alg Y}$. Since homomorphisms preserve terms, we have  $\Hom(\alg Y,\alg B) \subseteq \Hom(\alg X,\alg A)$ and the claim follows.

  For the classes $\Kpolyeff$ and $\Ksurjeff$, we take a reduct $\alg A'$ of $\alg A$ in a finite signature $S$ witnessing membership of $\alg A$ in the class. Then we take a finite signature $T$ containing all the symbols from the signature of $\alg{B}$ that appear in terms $t_f$, $f \in S$, and let $\alg{B}'$ be the corresponding reduct of $\alg{B}$. Given $\alg{Y}$ (in signature $T$) we compute $\alg{X}$ (in signature $S$) as above and get $\Hom(\alg Y,\alg B') \subseteq \Hom(\alg X,\alg A')$. We can then enumerate homomorphisms from $\alg Y$ to $\alg B'$ by enumerating homomorphism from $\alg X$ to $\alg A'$ and discarding the extra ones.
\end{proof}

For an algebra $\alg A$, let $\AddC{\alg A}$ denote the expansion of $\alg A$ by all the constant operations. Formally, for each $a \in A$, we add to the signature of $\alg A$ a fresh nullary symbol $a$ and interpret it in the obvious way $a^{\alg A} = a$. Informally,
$$
\AddC{\alg A} = \alg{A} \mbox{ plus constants}.
$$
Notice that every homomorphism to $\AddC{\alg A}$ is surjective.
Moreover, $\alg A$ and $\AddC{\alg A}$ have the same congruences, and the congruences of $\AddC{\alg A}$ are exactly those equivalence relations that are closed under applications of \emph{unary} term operations.

The next proposition shows that adding constants does not significantly decrease the surjective counting function.
The proof is simple, the fact is nevertheless crucial: in combination with the previous proposition, it allows us to concentrate on term operations of $\AddC{\alg A}$ rather than $\alg A$ and makes the tame congruence theory  directly applicable. We note that the standard name in universal algebra for a term operation of $\AddC{\alg A}$ is polynomial operation of $\alg A$. We refrain from using this terminology to avoid a potential confusion with polynomial functions.

\begin{proposition}\label{prop:Ksurj-polynomials}
    A finite algebra $\alg A$ is in $\Ksurj$ if, and only if 
    $\AddC{\alg A}$ is in $\Ksurj$.
    The same holds with $\Ksurjeff$ in place of $\Ksurj$.
\end{proposition}
\begin{proof} 
   The ``only if'' direction is clear.
   For the other direction, we concentrate on the efficient variant and assume for simplicity that $\alg A$ has a finite signature; the rest is an easy exercise. Let $A = \{a_1, \dots, a_k\}$. 

   Consider an input $\alg X$ in the signature of $\alg A$. For every choice of $\tuple{x} = (x_1, \dots, x_k) \in X^k$, we create an expansion $\alg X_{\tuple{x}}$ of $\alg X$ to the signature of $\AddC{\alg A}$ by defining $a_i^{\alg X_{\tuple{x}}} = x_i$. Since every homomorphism $h \colon \alg X \to \alg A$ such that $h(x_i)=a_i$ for all $i$ is a homomorphism $\alg X_{\tuple{x}} \to \AddC{\alg A}$, and every surjective homomorphism $h$ has this property for some $\tuple{x}$, we can enumerate all surjective homomorphisms $\alg{X} \to \alg{A}$ by enumerating all homomorphism $\alg{X}_{\tuple{x}} \to \AddC{\alg A}$ for every $\tuple{x}$. The number of choices for $\tuple{x}$ is $|X|^k$, so this procedure takes polynomial time. 
\end{proof}

The final proposition helps us to tame the input algebras $\alg X$. 

\begin{proposition}\label{prop:eq}
Let $\alg A$ be a finite algebra of finite signature and $\Sigma$ be a finite set of identities satisfied in $\alg A$.
For every finite algebra $\alg X$ in the signature of $\alg A$, there exists an algebra $\alg Y$ in the same signature and a surjective homomorphism $q\colon \alg X\to \alg Y$ such that $\alg Y$ satisfies $\Sigma$ and 
$$\Hom(\alg X,\alg A) = \Hom(\alg Y,\alg A) \circ q = \{h \circ q \mid h \in \Hom(\alg Y,\alg A)\}.
$$
Moreover, $\alg Y$ and $q$ can be computed from $\alg X$ in polynomial time.
\end{proposition}
\begin{proof}
Let $\alpha$ be the congruence of $\alg X$ generated by
$$
R = \{(s^{\alg X}(\tuple{z}),t^{\alg X}(\tuple z)) \mid
s(x_1, \dots, x_k) \approx t(x_1, \dots, x_k) \in \Sigma, \  \tuple{z} \in X^k\},
$$
let $\alg Y = \alg X/\alpha$, and let $q \colon \alg X \to \alg Y$ be the quotient homomorphism. Note that $R$ and $\alpha$, and then $\alg Y$ and $q$ can all be computed from $\alg X$ in polynomial time. From the definition of $R$ it follows that $\alg Y$ satisfies $\Sigma$. Clearly $\Hom(\alg X, \alg A) \supseteq \Hom(\alg Y,\alg A) \circ q$, so it only remains to verify the other inclusion.

Consider a homomorphism $h \colon \alg X \to \alg A$. For every $(s^{\alg X}(\tuple{z}),t^{\alg X}(\tuple z)) \in R$, we have $h(s^{\alg X}(\mbf z))=s^{\alg A}(h(\mbf z))=t^{\alg A}(h(\mbf z))=h(t^{\alg X}(\mbf z))$, therefore $R \subseteq \ker(h)$. Since $\ker(h)$ is a congruence of $\alg A$, we also have $\alpha \subseteq \ker(h)$, so $h$ indeed factorizes as $h = h' \circ q$ for the homomorphism $h' \colon \alg Y \to \alg A$ correctly defined by $h'(x/\alpha)=h(x)$. 
\end{proof}

In order to show that $\alg A \in \Kpolyeff$ for
a finite finite-signature algebra $\alg A$, we can now without loss of generality assume that the input algebra $\alg X$ satisfies any fixed finite set of identities satisfied by $\alg A$. Indeed, we compute $\alg Y$ and $q$ from the proposition, enumerate homomorphisms from $\alg Y$ to $\alg A$, and compose them with $q$.

The algebra $\alg Y = \alg X/\alpha$ produced in the proof makes sense for the infinite $\Sigma$ consisting of all identities satisfied by $\alg A$. However, we do not know whether it can be computed in polynomial time in general. 

\begin{question} \label{q:reflection}
For which finite algebras $\alg A$ of finite signature is there a polynomial-time algorithm that, given an input algebra $\alg X$ in the signature of $\alg A$, computes the smallest congruence $\alpha$ such that $\alg X/\alpha$ is in the variety generated by $\alg A$? 
\end{question}

The question is only interesting for algebras that are not \emph{finitely based}, i.e., for which there is no finite subset of identities satisfied by the algebra from which all the other satisfied identities follow. We refer to \cite{ALVIN-II} for a discussion about finite bases. 

\subsection{Examples} \label{subsec:examples}

The last proposition enables us to place any finite group to $\Kpolyeff$. We remark that for our proof of the main result it would be sufficient to consider commutative  groups of prime power order.

\begin{proposition}\label{prop:groups}
Every finite group is in $\Kpolyeff$.
\end{proposition}
\begin{proof}
   Let $\alg A$ be a group and
    let $\alg X$ be a finite algebra in the same signature. We can assume by~\Cref{prop:eq} (applied with $\Sigma_{grp}$ from (\ref{eq:Sigma_grp})) that
	$\alg X$ is a group.

    We start by greedily computing a generating set $Y$ of $\alg X$, i.e., we add a new element to $Y$ when it is not in the subgroup generated by the elements we already have. By Lagrange's theorem, $|Y| \leq \log_2 |X|$. There is $|A|^{|Y|} \leq |A|^{\log_2 |X|} = |X|^{\log_2|A|}$ mappings from $Y \to A$. For each of them, we try to extend it to a homomorphism $\alg X \to \alg A$. The extension may not exist but if it does, it is unique since $Y$ generates $\alg X$ and can be computed in polynomial time. This algorithm clearly enumerates all homomorphisms $\alg X \to \alg A$. 
\end{proof}

As we shall see in the second example, membership in $\Kpolyeff$ does not always come from small generating sets. Nevertheless, the following question is of independent interest. 

\begin{question} \label{q:generating}
    For which varieties does every finite member $\alg X$ have a generating set of size $O(\log |X|)$? 
\end{question}

We now turn to the second example. 
A \emph{semilattice} is an algebra in the signature consisting of a single binary symbol $\wedge$, which satisfies the identities 
    $$\Sigma =\{ x \wedge x \approx x, x \wedge y \approx y \wedge x, (x \wedge y) \wedge z \approx x \wedge (y \wedge z)\}.$$
For any semilattice $\alg X$, the binary relation $\leq_{\alg X}$ on $S$ defined by $x\leq_{\alg X} y$ if $x = x \wedge^{\alg X} y$ is a partial order on $X$. 

It is not hard to directly show that any finite semilattice is in $\Kpolyeff$; however, we only prove the result for two-element semilattices since this is what we will need; the result for arbitrary semilattices follows from~\Cref{cor:main}. 

\begin{proposition}\label{prop:semilattice}
	Every two-element semilattice  is in $\Kpolyeff$. 
\end{proposition}
\begin{proof} 
    Let $\alg A$ be a two-element semilattice. Without loss of generality, we assume that $A = \{0,1\}$ and $0 \leq_{\alg A} 1$, i.e.,
    $0 \wedge^{\alg A} 0 = 0 \wedge^{\alg A} 1 = 1 \wedge^{\alg A} 0 = 0$ and $1 \wedge^{\alg A} 1=1$. 
    Let $\alg X$ be a finite algebra in the same signature.  We can assume by
	\Cref{prop:eq} that $\alg X$ is a semilattice.
	
	For every homomorphism $h\colon\alg X\to\alg A$, the set $P = h^{-1}(\{1\})$ is upward closed and is closed under $\wedge^{\alg X}$.
	Indeed, if $x\in P$ and $x\leq_{\alg X} y$, then
	$$
 h(y) = 1\wedge^{\alg A} h(y) = h(x) \wedge^{\alg A} h(y)=h(x \wedge^{\alg X} y)=h(x)=1.$$
	Similarly, if $x,y\in P$, then $h(x \wedge ^{\alg X}y)=h(x)\wedge^{\alg{A}} h(y)=1$.
    It follows that $P=\emptyset$ or $P$ is the principal filter $p\!\!\uparrow\; =\{x\in X\mid x\geq_{\alg X} p\}$ where $p = \bigwedge P$.

    Homomorphisms from $\alg X$ to $\alg A$ can thus be enumerated by first listing the constant 0 mapping, and then iterating over the elements $p$ of $X$, building the corresponding principal filter $p\!\!\uparrow$ and checking whether the mapping that sends $x$ to $1$ iff $x\in p\!\!\uparrow$ is a homomorphism.
\end{proof}

Observe that the proof shows that $\counting{\alg A}(n) \leq n+1$. In fact, $\counting{\alg A}(n) = n+1$, as witnessed by an $n$-element semilattice $\alg X$ such that $\leq_{\alg X}$ is a linear order. 
Note also that semilattices in general do not have generating set of logarithmic size since the $(n+1)$-element semilattice $\alg X$ with universe $\{0,1, \dots, n\}$ such that $i \wedge^{\alg X} j = 0$ whenever $i \neq j$ cannot be generated by $n-1$ elements.

We also remark that using known results, the general facts, and the two examples, it is now quite easy to characterize two-element algebras in $\Kpoly$ (or the other three variants). Indeed, for any two-element algebra $\alg A$, the algebra $\AddC{\alg A}$ is term-equivalent to one of 7 specific algebras by, e.g., Lemma 4.8. in \cite{Hobby:1988}. Five of them have a term-reduct which is a semilattice or a group, so they are in $\Kpolyeff$ by~\Cref{prop:groups,prop:semilattice,prop:clone-Kpoly,{prop:Ksurj-polynomials}}. The remaining two contain only unary operations. It is a nice exercise to show that $c^s_{\alg A}$ grows exponentially for every unary algebra $\alg A$; the algebras one would construct to witness the exponential growth will likely be isomorphic, or at least similar, to those constructed in~\Cref{subsec:negative_result}.

\subsection{Strongly abelian congruences}
\label{subsec:strongly-abelian}

It is time to define strongly abelian congruences. 

\begin{definition}\label{def:strong-abelian}
    Let $\alg A$ be an algebra.
    A congruence $\alpha$ of $\alg A$ is \emph{strongly abelian} if for all $t\in\Clo_k(\alg A)$ ($k\geq 1$) and all $x_1, \dots, x_k$, $y_1, \dots, y_k$, and $z_2, \dots, z_k$ in $A$ such that $(x_i,y_i)\in\alpha$ for all $i\geq 1$ (and $i \leq k$) and $(y_i,z_i)\in\alpha$ for all $i\geq 2$, we have
   $$  t(x_1,x_2,\dots,x_k)=t(y_1,y_2,\dots,y_k)
   \ \mbox{ implies } \  t(x_1,z_2,\dots,z_k)=t(y_1,z_2,\dots,z_k).
   $$
\end{definition}

We make several simple observations. First, $0_A$ is trivially a strongly abelian congruence of any algebra $\alg A$. Other strongly abelian congruences are called \emph{nontrivial}. Second, if $\alpha$ is a strongly abelian congruence of $\alg A$, then so is any congruence $\beta \subseteq \alpha$. In particular, if $\alg A$ has a nontrivial strongly abelian congruence, then it has a minimal one. Third, if $\alpha$ is a strongly abelian congruence of $\alg A$ and $\alg B$ is a subalgebra of $\alg A$, then $\alpha \cap (B \times B)$ is a strongly abelian congruence of $\alg B$. Fourth, a strongly abelian congruence of $\alg A$ is also a strongly abelian congruence of $\AddC{\alg A}$. 

An algebra is \emph{strongly abelian} if $1_A$ is strongly abelian, equivalently, all congruences are strongly abelian. 
Examples of strongly abelian algebras include essentially unary algebras. The definitions are as follows. We say that an operation $t$ on $A$, or more generally a mapping $t: A^k \to B$, \emph{depends} on the $i$th coordinate (where $1 \leq i \leq k$) if there exists $a_1, \dots a_k$ and $a' \in A$ such that $t(a_1, \dots, a_k) \neq t(a_1, \dots, a_{i-1},a',a_{i+1}, \dots, a_k)$. We call $t$ \emph{essentially unary} if it depends on at most one coordinate. An algebra is \emph{essentially unary} if so is each of its basic (term) operation. 

Note that, indeed, every essentially unary algebra is strongly abelian; different examples are provided in~\Cref{subsec:matrix-power}.
On the other hand, having a strongly abelian congruence can be seen as  ``being locally close to unary'' by reading the contrapositive of the implication in~\Cref{def:strong-abelian}: if $t$ depends on the first coordinate and this is witnessed by suitable $x_1$, $y_1$ and the $z_i$, then $t(x_1, \dots)$ is ``often'' different from $t(y_1, \dots)$ (which would always be the case if $t$ depended only on the first coordinate).  

Having a nontrivial strongly abelian congruence can be regarded as a rather pathological situation, e.g., no group or semilattice (nor an expansion thereof, e.g., a ring, a module, or a lattice) have this property, and the only such two-element algebras are essentially unary.
In the mentioned examples, even no quotients have nontrivial strongly abelian congruences. The following example shows that latter property is strictly stronger and shows that the classes $\Kpoly$, etc. are not closed under quotients.

\begin{example}\label{ex:not-closed-quotient}
Let $\alg A$ be the algebra on $\{0,1,2\}$ with 
a single binary operation $\cdot^{\alg A} = \cdot$ defined by the following table.

\[
\begin{array}{r|ccc}
\cdot & 0 & 1 & 2\\ 
\hline
0 & 0 & 0 & 2\\
1 & 0 & 1 & 2 \\
2 & 1 & 0 & 2
\end{array}
\]

Let $\alpha$ be $\{0,1\}^2\cup \{2\}^2$. One can see that no proper subalgebra of $\alg A$ has a nontrivial strongly abelian congruence, that $\alpha$ is the only congruence of $\alg A$ different from $0_A$ and $1_A$, and that $\alg A/\alpha$ is an essentially unary 2-element algebra, so $1_{A/\alpha}$ is a strongly abelian congruence of $\alg A/\alpha$. 
However, note that $\alpha$ is not strongly abelian since $0\cdot 1 = 1\cdot 0$, but $0\cdot 1 \neq 1\cdot 1$. 
It follows from~\Cref{cor:main,cor:main-surj}  that $\alg A$ is in $\Kpolyeff$ while $\alg A/\alpha$ is not even in $\Ksurj$. 

On the other hand, it can be directly verified that $\Kpoly$ and $\Kpolyeff$ are closed under finite products of algebras with the same signature and under subalgebras,  and that $\Ksurj$ and $\Ksurjeff$ are closed under finite products. The latter two classes are not closed under subalgebras -- consider the algebra $\alg B$ with the same universe $B=\{0,1,2\}$ and a binary operation defined as follows. 

\[
\begin{array}{r|ccc}
\cdot & 0 & 1 & 2\\ 
\hline
0 & 0 & 1 & 0\\
1 & 0 & 1 & 0 \\
2 & 0 & 2 & 2
\end{array}
\]
The subalgebra $\alg C$ of $\alg B$ with universe $\{0,1\}$ is essentially unary, so $\counting{\alg C}^s$ grows exponentially. On the other hand, $\alg B$ is simple and is not strongly abelian (e.g., $0 \cdot 2 = 2 \cdot 0$ while $0 \cdot 2 \neq 2 \cdot 2$), therefore $\alg{B} \in \Ksurjeff$ by~\Cref{cor:main-surj}.
\end{example}

The final goal of this section is to show that the existence of a nontrivial  strongly abelian congruence in a finite-signature algebra $\alg A$ (or some if its subalgebras) can be decided in polynomial time. In particular, the condition in item (\ref{itm:main-official-strongly-abelian}) of \Cref{thm:main-official} can be efficiently checked. We employ the following consequence of Theorem 7.2 in \cite{Hobby:1988}.

\begin{theorem}[\cite{Hobby:1988}] \label{thm:snags}
    Let $\alpha$ be a congruence of a finite algebra $\alg A$. The following are equivalent.
    \begin{enumerate}
    \item The congruence $\alpha$ is strongly abelian,
    \item There do not exist distinct $a,b \in A$ and $t \in \Clo_2(\AddC{\alg A})$ such that $(a,b) \in \alpha$, $t(a,b)=t(b,a)=a$, and $t(b,b)=b$ (cf.~\Cref{ex:not-closed-quotient}). 
    \end{enumerate}
\end{theorem}

\begin{proposition} \label{prop:efficient-decision} 
There exist polynomial-time algorithms that, given a finite algebra $\alg A$ of finite signature, decide whether
\begin{enumerate}
    \item \label{itm:efficient-decision-alg} $\alg A$ has a nontrivial strongly abelian congruence. 
    \item \label{itm:efficient-decision-subalg} some subalgebra of $\alg A$ has a nontrivial strongly abelian congruence. 
\end{enumerate}
\end{proposition}

\begin{proof} 
  We first observe that, for every $a,b \in A$, the following conditions are equivalent.
  \begin{itemize}
      \item  There exists $t \in \Clo_2(\AddC{\alg A})$ such that $(a,b) \in \alpha$, $t(a,b)=t(b,a)=a$, and $t(b,b)=b$.
      \item  The subalgebra of $(\AddC{\alg A})^3$ generated by $X = \{(a,b,b),(b,a,b)\}$ contains $(a,a,b)$. 
  \end{itemize}
The reason is that the universe of a subalgebra generated by $X$ is equal to the one-step closure of $X$ under term operations. Therefore, $(a,a,b)$ is in the subalgebra appearing in the second item if, and only if, there exists $t \in \Clo_2(\alg B)$, where $\alg B = (\AddC{\alg A})^3$, such that $t^{\alg B}((a,b,b),(b,a,b))=(a,a,b)$, which is equivalent to the first item.
Moreover, the second condition can be verified in polynomial time because that universe can also be obtained as a (many-step) closure under basic operations. 
By~\Cref{thm:snags}, we can thus check in polynomial-time whether a given congruence is strongly abelian by going over all pairs of distinct elements $(a,b) \in \alpha$. 

Suppose that some subalgebra $\alg B$ of $\alg A$ has a nontrivial strongly abelian congruence $\alpha$. Let $(c,d) \in \alpha$ and $c \neq d$. By the remarks following~\Cref{def:strong-abelian}, in the subalgebra of $\alg A$ generated by $\{c,d\}$ (which is contained in $B$), the congruence generated by $\{(c,d)\}$ is  strongly abelian. Therefore, in order to verify whether some subalgebra $\alg A$ has a nontrivial strongly abelian congruence, it is enough to concentrate on these subalgebras and congruences. Since there are fewer than $|A|^2$ such situations and generating subalgebras or congruences can be done in polynomial time, item~(\ref{itm:efficient-decision-subalg})  follows. Item~(\ref{itm:efficient-decision-alg}) is similar: it is enough to check whether congruences generated by pairs of distinct elements are strongly abelian.
\end{proof}

\section{Non-membership} \label{sec:nonmember}

This section is devoted to the simpler, negative part of our main result. In~\Cref{subsec:negative_result}, we provide a superpolynomial lower bound on $\counting{\alg A}^s$  in case that $\alg A$ has a nontrivial strongly abelian congruence. \Cref{subsec:matrix-power} shows that the obtained bound is essentially optimal.

\subsection{Superpolynomial lower bound} \label{subsec:negative_result}

Let $\alg A$ be a finite algebra.
If some algebra $\alg F$ in the signature of $\alg A$ and a subset $X \subseteq F$ have the property that
every mapping $X \to A$ can be extended to a homomorphism $\alg F \to \alg A$, then we get $\counting{\alg A}(|F|) \geq |A|^{|X|}$. 
Therefore, if we have such an $\alg F$ for every finite $X$ and $|F|$ is upper bounded by a polynomial in $|X|$, then we get a superpolynomial lower bound on $\counting{\alg A}$. 

Examples of algebras $\alg F$ with the above extension property with respect to $X \subseteq F$ include the \emph{free algebra  over $X$ in the variety generated by $\alg A$}. This algebra, denote it $\alg F_{\alg A}(X)$, can be defined as the quotient of the algebra of all terms over $X$ by the congruence consisting of identities satisfied in $\alg A$. A convenient alternative description for finite $X$ is as follows. Identify $X$ with the set $\{\pi^n_1, \dots, \pi^n_{n}\}$, where $n = |X|$ and $\pi^n_i: A^n \to A$ denotes the $n$-ary projection to the $i$th variable, i.e., $\pi^n_i(a_1, \dots, a_n) = a_i$. The free algebra $\alg F_{\alg A}(X)$ can be defined as the subalgebra of $\alg A^{A^{n}}$ with universe $\Clo_n(\alg A)$ (recall here that the universe of $\alg A^{A^n}$ is $A^{A^n}$, the set of mappings $A^n \to A$, so the definition makes formal sense). 

Corollary 4.8 in~\cite{finite-complexity} 
characterizes those finite algebras $\alg A$ for which $|F_{\alg{A}}(X)|$ depends polynomially on $|X|$. By the discussion above, such algebras have a superpolynomial lower bound on $\counting{\alg A}$.
The following example shows that these considerations are not sufficient for our purposes.

\begin{example} 
The algebra $\alg A$ with universe $\{0,1,2\}$ and a single binary operation defined by 
\[
\begin{array}{r|ccc}
\cdot & 0 & 1 & 2\\ 
\hline
0 & 0 & 1 & 0\\
1 & 0 & 1 & 0 \\
2 & 0 & 1 & 2
\end{array}
\]
has a nontrivial strongly abelian congruence, namely  $\{0,1\}^2 \cup \{2\}^2$ (this follows e.g. from a simple calculation of term operations), but also a semilatice quotient modulo the same congruence. Therefore, $\counting{\alg A}$ has a superpolynomial growth by~\Cref{thm:main-official}, but $|F_{\alg{A}}(X)|$ is not bounded by a polynomial since even the two-element semilattice quotient has exponentially many term operations.
\end{example}

 However, a natural modification of free algebras turns out to work for our purposes.

\begin{definition}\label{def:ab-free}
    Let $\alg A$ be an algebra and $a,b \in A$. 
    Let $\alg F$ be an algebra in the same signature and $X\subseteq F$.
    We say that $\alg F$ is \emph{$ab$-free for $\alg A$ with $ab$-free set $X$} if every mapping $X\to\{a,b\}$ can be extended to a homomorphism $\alg F\to \alg A$.
\end{definition}

The following proposition shows that $ab$-free algebras can be constructed similarly to free algebras, with the difference that we restrict all term operations to $\{a,b\}$. (Naturally, one can  similarly define $B$-free for any subset $B$ of $A$, not just $B = \{a,b\}$, and prove a similar result.)

\begin{proposition} \label{prop:abfree}
Let $\alg A$ be an algebra and $a,b \in A$. Let $\alg F_{\alg A,ab}(n)$ be the subalgebra of $\alg A^{\{a,b\}^n}$ with universe 
$$
F_{\alg A,ab}(n) = \{t|_{\{a,b\}^n} \colon \{a,b\}^n \to A \mid  t \in \Clo_n(\alg A)\}.
$$
The algebra $\alg F_{\alg A,ab}(n)$ is $ab$-free for $\alg A$ with $ab$-free set $X = \{\pi^n_1|_{\{a,b\}^n}, \dots, \pi^n_n|_{\{a,b\}^n}\}$.
\end{proposition}

\begin{proof}
    First observe that $F_{\alg A,ab}$ is preserved by every basic operation of $\alg A^{\{a,b\}^n}$: for every $f$ in the signature of $\alg A$ of arity $k$, and every $t_1,\dots,t_k\in\Clo_n(\alg A)$, we have that 
    $$
    f^{\alg A^{\{a,b\}^n}}(t_1|_{\{a,b\}^n},\dots,t_k|_{\{a,b\}^n}) = \left(f^{\alg A}(t_1,\dots,t_k)\right)|_{\{a,b\}^n}, 
    $$
    where
    $$
    \left(f^{\alg A}(t_1,\dots,t_k)\right) (a_1, \dots, a_n) 
    = 
    f^{\alg A}\left(t_1(a_1, \dots, a_n), \dots, t_k(a_1, \dots, a_n)\right).
    $$
    Therefore, the definition of $\alg F_{\alg A,ab}(n)$ makes sense.
    Denote in the following $\pi^n_i|_{\{a,b\}^n}$ by $x_i$.
    Given any mapping $h: X \to \{a,b\}$, we define $h': F_{\alg A,a,b}(n) \to A$ by $$
    h'(t|_{\{a,b\}^n}) = t(h(x_1),\dots,h(x_n)),
    $$
    which is well-defined for if $t|_{\{a,b\}^n}=s|_{\{a,b\}^n}$, then $t(h(x_1),\dots,h(x_n))$ $=$ $s(h(x_1),\dots,h(x_n))$.
    It clearly extends $h$, and it remains to show that $h'$ is a homomorphism from $\alg F_{\alg A,ab}(n)$ to $\alg A$.

    Let $f$ be a symbol of arity $k$ in the signature of $\alg A$, and let $t_1,\dots,t_k\in\Clo_n(\alg A)$.
    Then
    \begin{align*}
    h'\left(f^{\alg F_{\alg A,ab}(n)}(t_1|_{\{a,b\}^n},\dots,t_k|_{\{a,b\}^n})\right) &= h'\left(\left(f^{\alg A}(t_1,\dots,t_k)\right)|_{\{a,b\}^n}\right)
    \\&= \left(f^{\alg A}(t_1,\dots,t_k)\right)(h(x_1),\dots,h(x_n))
    \\&= f^{\alg A}(t_1(h(x_1),\dots,h(x_n)),\dots,t_k(h(x_1),\dots,h(x_n)))
    \\&=f^{\alg A}(h'(t_1|_{\{a,b\}^n}),\dots,h'(t_k|_{\{a,b\}^n})),
    \end{align*}
    so that $h'$ is indeed a homomorphism $\alg F_{\alg A,ab}(n)\to\alg A$.
\end{proof}

We now show that the $ab$-free algebras we just constructed have polynomial size whenever $(a,b)$ is in a strongly abelian congruence.

\begin{proposition} \label{prop:abfree-are-small}
    Let $\alg A$ be an algebra with strongly abelian congruence $\alpha$ and let $a,b \in A$ be such that $(a,b) \in \alpha$.
    Then $|F_{\alg A,ab}(n)| \in O(n^k)$ where $k = \lfloor \log_2 |A| \rfloor$.
\end{proposition}

\begin{proof}
Consider $t \in \Clo_n(\alg A)$ and its restriction $s = t|_{\{a,b\}^n}$. Let $I$ be the set of coordinates on which $s$ depends.  We show that $|I| \leq k$. 
If $1 \in I$, then there exist $z_2, \dots, z_n \in \{a,b\}$ such that $s(a,z_2, \dots, z_n) \neq s(b, z_2, \dots, z_n)$. Since $\alpha$ is strongly abelian, we must then have $s(a,x_2, \dots, x_n) \neq s(b, y_2, \dots, y_n)$ for any $x_2, \dots, x_n, y_2, \dots, y_n \in \{a,b\}$. In other words, whenever two tuples $\tuple{x},\tuple{y} \in \{a,b\}^n$ differ on the first coordinate, the results $s(\tuple{x})$ and $s(\tuple{y})$ are different as well. Similar conclusion can be derived for any coordinate $i \in I$ by using the property in \Cref{def:strong-abelian} to $s$ with permuted coordinates. Applying $s$ to any $2^{|I|}$-element set of arguments that are pairwise different on $I$, we get  that $s$ attains at least $2^{|I|}$ values. So $2^{|I|} \leq |A|$ and, indeed, $|I| \leq k$.

Every element $s\in\alg F_{\alg A,ab}(n)$ is thus given by the choice of a $k$-element subset of coordinates $I$ and  a mapping from $\{a,b\}^I$ to $A$.
    Therefore, $|F_{\alg A, ab}(n)|$ is at most $n^k |A|^{2^k} \in O(n^k)$.
\end{proof}

The main result of this section now follows by a simple calculation.

\begin{corollary}\label{cor:nonmembership}
    Let $\alg A$ be a finite algebra.
    If $\alg A$ has a nontrivial strongly abelian congruence,
    then $\counting{\alg A}^s(n) \in 2^{\Omega(n^{1/k})}$ where $k = \lfloor \log_2 |A| \rfloor$. 
\end{corollary}

\begin{proof}
Let $\alg B = \AddC{\alg A}$ and
let $a,b$ be distinct elements such that $(a,b) \in  \alpha$. 
Denote $d(m) = |F_{\alg B,ab}(m)|$. By \Cref{prop:abfree}, $\alg F_{\alg B,ab}(m)$ is $ab$-free for $\alg B$ with $ab$-free set of size $m$, therefore $\counting{\alg B}(d(m)) \geq 2^m$. Since $\alpha$ is still a strongly abelian congruence of $\alg B$, \Cref{prop:abfree-are-small} implies that $d(m) \leq Cm^k$ for some  constant $C$.
Let $n$ be a sufficiently large integer and let $m$ be the largest integer such that $Cm^k \leq n$. We have $m \geq C'n^{1/k}$ for a suitable positive constant $C'$, so  
$$
\counting{\alg A}^s(n) \geq \counting{\alg B}^s(n) = \counting{\alg B}(n) \geq \counting{\alg B}(Cm^k) \geq \counting{\alg B}(d(m)) \geq 2^m \geq 2^{C'n^{1/k}}
    $$ 
and the claim follows.    
\end{proof}

The proof in this section is based on a modification of the standard and useful free algebra concept. We wonder whether these ``somewhat free'' algebras (or other structures) naturally occur elsewhere as well.

\subsection{Matrix powers of sets} \label{subsec:matrix-power}

We have just shown that $\counting{\alg{A}}^s(n) \geq 2^{Cn^{1/k}}$, where $k = \lfloor \log_2 |A| \rfloor$ and $C$ is a positive constant, whenever $\alg A$ has a nontrivial strongly abelian congruence. In this section we show that, for each positive integer $k$, there is a strongly abelian algebra $\alg A$ of size $|A|=2^k$ such that $\counting{\alg A}(n) \leq 2^{n^{1/k}}$. 

The construction is a special case of so-called matrix power of an algebra. The algebras we use are matrix powers of algebras with no operations --- sets. The concept of a matrix power emerged independently in various contexts, we refer to Section 10.6 of \cite{ALVIN-III} where the general construction is discussed. 

Fix a positive integer $k$ and let $S$ be the signature consisting of a unary symbol $s$ (for \emph{s}hift) and a $k$-ary symbol $d$ (for \emph{d}iagonal). Let $\Sigma$ be the following set of identities, where $s^k(x)$ should be read as $s(s(\dots (s(x))\dots)$ with $k$ occurences of $s$.
\begin{align*}
d(x, \dots, x) & \approx x \\
                   s^k(x)  & \approx x \\
                   s(d(x_1, \dots, x_k)) & \approx d(s(x_k),s(x_1), \dots, s(x_{k-1})) \\
                   d(d(x_{11}, \dots, x_{1k}), d(x_{21}, \dots, x_{2k}), \dots, d(x_{k1}, \dots, x_{kk})) & \approx d(x_{11}, x_{22}, \dots, x_{kk}) 
\end{align*}
For a set $Y$, we denote by $Y^{[k]}$ the algebra with universe $Y^k$ and basic operations defined as follows.
\begin{align*}
s^{Y^{[k]}}((y_1, y_2, \dots, y_k)) &= (y_2, \dots, y_k,y_1) \\
d^{Y^{[k]}}((y^1_1, y^1_2, \dots, y^1_k), \dots, (y^k_1, \dots, y^k_k)) &= (y^1_1, y^2_2, \dots, y^k_k)
\end{align*}

The following two facts imply that the algebras $Y^{[k]}$ are fully axiomatized by $\Sigma$ and give us a full understanding on homomorphisms between them. We do not attribute them to any specific set of authors for the reason above. We give brief sketches of proofs, full proofs are given in \cite[Theorem 10.92, Theorem 10.98]{ALVIN-III}. 

\begin{proposition} \label{prop:matrix-powers-axiomatized}
An algebra $\alg A$ in signature $S$ satisfies $\Sigma$ if, and only if, $\alg A$ is isomorphic to $Y^{[k]}$ for some set $Y$. 
\end{proposition}
\begin{proof}[Proof sketch]
    The backward implication amounts to checking that every $Y^{[k]}$ satisfies $\Sigma$, which is straightforward.

    Assume now that $\alg A$ satisfies $\Sigma$. Let $Y = \{ a \in A \mid s^{\alg A}(a)=a\}$ and define mappings $h: Y^k \to A$ and $h': A \to Y^k$ as follows (omitting the superscripts $\alg A$).
    \begin{align*}
    h(y_1, \dots, y_k) &= d(y_1, \dots, y_k) \\
    h'(a) &= (d(a, s(a), \dots, s^{k-1}(a)), \\
                & \quad\quad d(s^{k-1}(a), a, \dots, s^{k-2}(a)), \dots, \\
                & \quad\quad d(s(a),  \dots, s^{k-1}(a), a))
    \end{align*}
    Using the identities in $\Sigma$ it can be verified that 
    the mapping $h'$ is correctly defined, that the composition $hh'$ is the identity on $A$, that $h'h$ is the identity on $Y^k$, and that $h$ preserves $d$ and $s$, so $h$ is an isomorphism $\alg A \to Y^{[k]}$. 
\end{proof}

\begin{proposition} \label{prop:matrix-power-homo}
   Let $Y, Z$ be sets. For any mapping $h:Z \to Y$, the induced mapping $h^k: Z^k \to Y^k$ is a homomorphism from $Z^{[k]}$ to $Y^{[k]}$. Conversely, every homomorphism from $Z^{[k]}$ to $Y^{[k]}$ has this form.
\end{proposition}

\begin{proof}[Proof sketch]   The first part is straightforward. For the second part, observe that the diagonal elements $\bar{z} = (z,\dots, z)$ in $Z^{[k]}$ are characterized by the property $s(\bar{z})=\bar{z}$. A homomorphism $h'$ from $Z^{[k]}$ to $Y^{[k]}$ therefore maps the diagonal elements of $Z^{[k]}$ to diagonal elements of $Y^{[k]}$, and thus induces a mapping $h \colon Z \to Y$. That $h'=h^k$ then follows from preservation of $d$. 
\end{proof}

The algebra promised in the beginning of the subsection is $\alg A = \{1,2\}^{[k]}$. It essentially only remains to observe that matrix powers of sets are strongly abelian. 

\begin{corollary} \label{cor:matrix-powers}
    For any finite $Y$, the algebra $\alg A = Y^{[k]}$ is strongly abelian and $\counting{\alg A}(n) \leq |Y|^{n^{1/k}}$ for every $n$. 
\end{corollary}

\begin{proof} 
   Strong abelianess follows from the definition and a description of term operations of $\alg A$: they are exactly operations of the form 
   $$
   t((x^1_1, \dots, x^1_k), (x^2_1, \dots, x^2_k), \dots, (x^n_1, \dots, x^n_k)) = (x^{i_1}_{j_1}, x^{i_2}_{j_2}, \dots, x^{i_n}_{j_n})
   $$
   for some $n$ and $i_1, \dots, i_n$, and $j_1, \dots, j_n$.

   Let $\alg X$ be any algebra in the signature $S$ and let $\alg X'$ be the algebra from  \Cref{prop:eq} such that $\alg X'$ satisfies $\Sigma$, $|X'| \leq |X|$, and $|\Hom(\alg X,\alg A)| \leq |\Hom(\alg X',\alg A)|$.  By \Cref{prop:matrix-powers-axiomatized}, the algebra $\alg X'$ is isomorphic to $Z^{[k]}$ for some $Z$.  Additionally applying \Cref{prop:matrix-power-homo}, we obtain
   $$
   |\Hom(\alg X,\alg A)| \leq
   |\Hom(Z^{[k]},Y^{[k]})| =
   |Y|^{|Z|} \leq |Y|^{|X|^{1/k}}
   $$
   and the claim follows.
\end{proof}

\section{Membership} \label{sec:member}

This section provides the implication from (\ref{itm:main-official-strongly-abelian}) to (\ref{itm:main-official-counting-poly}) in~\Cref{thm:main-official}; more precisely, we show that $\alg A \in \Ksurjeff$ whenever $\alg A$ has no minimal strongly abelian congruence. The necessary prerequisites from tame congruence theory are reviewed in~\Cref{subsec:tct}. The proof is given in~\Cref{subsec:suff}.

\subsection{Tame congruence theory}\label{subsec:tct}

For an algebra $\alg A$ 
and a subset $N \subseteq A$ we denote by $\alg A|_N$ the \emph{algebra induced by $\alg A$ on $N$}, that is, the algebra with universe $N$ whose operations are those term operations of $\AddC{\alg A}$ that preserve $N$ (the signature of this algebra can be chosen arbitrarily so that the operation symbols are in bijective correspondence with the operations of $\alg A|_N$).

One of the core components of the theory is the classification of so-called minimal algebras.
An algebra $\alg M$ is \emph{minimal} if $|M| \geq 2$ and every unary term operation of $\AddC{\alg M}$ is either a constant or a permutation.

\begin{theorem}[\cite{Palfy}]\label{thm:palfy}
    Let $\alg M$ be a finite minimal algebra. Then $\AddC{\alg M}$ is term-equivalent to $\AddC{\alg A}$, where $\alg A$ is exactly one of the following:
    \begin{enumerate}
        \item An algebra with only unary operations, all of which are permutations.
        \item A vector space over a finite field.  
        \item A two-element boolean algebra.
        \item A two-element lattice.
        \item A two-element semilattice. 
    \end{enumerate}
\end{theorem}

\noindent
A minimal algebra is said to have type $i$, for $i\in\{1,\dots,5\}$, if item ($i$) in~\Cref{thm:palfy} takes place. 
Note that if $\alg M$ is a finite minimal algebra of type 2--5, then there exists an algebra $\alg A$ such that $\Clo(\alg A)\subseteq\Clo(\AddC{\alg M})$ and $\alg A$ is either a group or a 2-element semilattice.

Consider now an arbitrary finite algebra $\alg A$, not necessarily minimal. 
A pair $(\gamma,\delta)$ of congruences of $\alg A$ is called a \emph{cover} if $\gamma\subsetneq \delta$ and no congruence $\theta$ of $\alg A$ lies strictly between $\gamma$ and $\delta$.
By using a construction explained below, for each cover $(\gamma,\delta)$ one obtains a minimal algebra that belongs to one of the five classes in the previous theorem and can therefore be given a type $i\in\{1,\dots,5\}$.
We then assign this type to the pair $(\gamma,\delta)$,
which we write $\typ_{\alg{A}}(\gamma,\delta)=i$ or just $\typ(\gamma,\delta)$ if $\alg A$ is clear from the context.

The construction goes as follows. Let $U$ be a minimal set with the property that there exists a unary term operation $p \in \Clo_1(\AddC{\alg A})$ such that $p(A)=U$ and $p(\delta)\not\subseteq\gamma$; such a set is called a \emph{$(\gamma,\delta)$-minimal set} of $\alg A$. 
Then, for every $a\in U$ such that $a/\delta\cap U\not\subseteq a/\gamma$,
the induced algebra $\alg A|_N$ on the set $N=(a/\delta\cap U)$ is such that $(\alg A|_N)/\gamma$ is a minimal algebra~\cite[Lemma 2.16 (3)]{Hobby:1988} that is of one of the 5 types given in~\Cref{thm:palfy}. 
Such a set $N$ is called a \emph{$(\gamma,\delta)$-trace}.
All the $(\gamma,\delta)$-traces have the same type independent on the chosen $U$ and $N$~\cite[Theorem 4.23]{Hobby:1988}, and therefore $\typ(\gamma,\delta)$ is well-defined.

Corollary 5.3 in~\cite{Hobby:1988} shows that the types behaves well with respect to quotients.

\begin{theorem}[\cite{Hobby:1988}]\label{thm:type-quotient}
Let $\alg A$ be a finite algebra and $\beta \subseteq \gamma \subseteq \delta$ be congruences such that $(\gamma,\delta)$ is a cover. Then
$$
\typ_{\alg A}(\gamma,\delta)
=\typ_{\alg A/\beta}(\gamma/\beta,\delta/\beta).
$$
\end{theorem}

Chapter 9 of \cite{Hobby:1988} provides several results characterizing omitting types in all finite algebras in a variety by means of identities. We will only need a very  special case of~\cite[Theorem 9.6]{Hobby:1988}. Two concepts are required to formulate it: an operation $t: A^n \to A$ is \emph{idempotent} if $t(a,a, \dots, a)=a$ for every $a \in A$; a ternary operation $t: A^3 \to A$ is \emph{Mal'cev} if $t(b,a,a)=b=t(a,a,b)$ for every $a,b \in A$. 

\begin{theorem}[\cite{Hobby:1988}] \label{thm:taylor-and-types}
Let $\alg A$ be a finite algebra that has a ternary Mal'cev term operation or a binary commutative idempotent  term operation. Then $\typ(\gamma,\delta) \neq 1$ for every cover of congruences $(\gamma,\delta)$ in $\alg A$. 
\end{theorem}

We now concentrate only on the case $(\gamma,\delta) = (0_{A},\alpha)$ for a minimal congruence $\alpha$. Note 
that traces in a $(0_A, \alpha)$-minimal set $U$ are exactly those intersections of $\alpha$-equivalence classes with $U$ that have at least 2 elements. 
The union  of traces, denote it $B$, is referred to as the \emph{body} of $U$ and  the complement $T = U \setminus B$ is the \emph{trace} of $U$.

Theorem 2.8 in~\cite{Hobby:1988} shows various density and separation properties of minimal sets. Two of them (items (2) and (4) in that theorem) are crucial for us. 

\begin{theorem}[\cite{Hobby:1988}]\label{thm:HM-separation}
Let $\alg A$ be a finite algebra, $\alpha$ be a minimal congruence, and $U$ be a $(0_A,\alpha)$-minimal set. 
\begin{enumerate}
    \item\label{itm:idempotent-minimal} There exists $e \in \Clo_1(\AddC{\alg A})$  such that $e|_{U} = \mathrm{id}_U$  and $e(A) = U$ 
    \item\label{itm:separation} For every $(a,b) \in \alpha$ with $a \neq b$, there exists $e \in \Clo_1(\AddC{\alg A})$ such that $e(a)\neq e(b)$ and $e(A)=U$.
\end{enumerate}
\end{theorem}

The following direct relation of types to strong abelianness  is a consequence of Theorem 5.6 in \cite{Hobby:1988}.

\begin{theorem}[\cite{Hobby:1988}] \label{thm:strongly-abelian-omit-one}
A minimal congruence $\alpha$ of a finite algebra $\alg A$ is strongly abelian if, and only if, $\typ(0_A,\alpha) \neq 1$. 
\end{theorem}

Therefore, we are interested in algebras such that $\typ(0_A,\alpha) \neq 1$ for every minimal congruence $\alpha$. Strong structure theorems are available for $\typ(0_A,\alpha)$-minimal sets in this case. We will only need several pieces of  information that follow from~\cite{Hobby:1988,Easy} and some additional work.

\begin{theorem}
\label{thm:pseudo-malcev-meet}
Let $\alg A$ be a finite algebra, let $\alpha$ be a minimal congruence such that $\typ(0_A, \alpha) \neq 1$, and let $U$ be a $(0_A,\alpha)$-minimal set with body $B$ and tail $T$. Then
\begin{enumerate}
   \item \label{itm:body-is-nice} $\alg A|_B$ has a Mal'cev term operation or a binary commutative idempotent term operation, and
   \item \label{itm:body-and-tail} $\AddC{\alg A}$ has a binary term operation $p$ such that 
   $$p(B,B) \subseteq B, \ p(B,T) \subseteq T, \ p(T,B) \subseteq T, \mbox{ and } p(T,T) \subseteq T.
   $$
\end{enumerate}
\end{theorem}

\begin{proof}
    If $\typ(0_A,\alpha) \in \{3,4,5\}$, then  Lemmas 4.15 and 4.17 in~\cite{Hobby:1988} (or Lemma 3.2 in~\cite{Easy}) imply that $B$ is a two-element set, say $B=\{0,1\}$, and there exists $p \in \Clo_2(\alg A|_U)$ such that 
    \begin{itemize}
        \item $p(u,1)=p(1,u)=p(u,u)=u$ for every $u \in U$,
        \item $p(u,0)=p(0,u)=u$ for every $u \in U \setminus \{1\}$, and
        \item $p(u,p(u,v))=p(u,v)$ for every $u,v \in U$.
    \end{itemize}
    It follows that the restriction of $p$ to $B$  is commutative and idempotent (indeed, it is a semilattice operation on the two-element set $B$) and  $p$ satisfies the first three inclusions in the second item. It also satisfies the last one, since if $p(u,v) \not\in T$ for some $u,v \in T$, then $p(u,v) \in B$ and we get $p(u,v)=p(u,p(u,v)) \in p(T,B) \subseteq T$, a contradiction. 

    Assume now that $\typ(0_A,\alpha)=2$. By Lemma 4.20 in~\cite{Hobby:1988} (or Lemma 3.5 in \cite{Easy}) and Lemma 3.6 in \cite{Easy}, there exists a ternary operation $d \in \Clo_3(\alg A|_U)$ such that
    \begin{itemize}
        \item $d(u,u,u)=u$ for every $u \in U$,
        \item $d(b,b,u)=u=d(u,b,b)$ for every $b \in B$, $u \in U$,
        \item $d(t,t,b)\in T$ for every $t\in T$, $b\in B$, and
        \item $B$ is closed under $d$.
    \end{itemize}
    The restriction of $d$ to $B$ is thus a Mal'cev operation in $\Clo(\alg A|_B)$. In order to find $p$ satisfying the second item, we start with $p_0$ defined by $$p_0(u,v) = d(u,u,v)$$ and inductively define
    $$p_{i+1}(u,v) = p_0(p_i(u,v),v).
    $$
    Observe that $p_0(U,U) \subseteq U$, $p_0(B,B) \subseteq B$, $p_0(T,B) \subseteq T$ and, by induction, these inclusions hold for any $p_i$. It is thus enough to find $i$ such that $p_i(U,T) \subseteq T$.
    
    Fix $u \in U$ and $t \in T$ and consider the sequence
    $$
    p_0(u,t), \ p_1(u,t), \ p_2(u,t), \dots
    $$
    If some member of this sequence, say $b=p_j(u,t)$, is not in $T$, then it is in $B$, thus $p_{j+1}(u,t) = p_0(b,t) = d(b,b,t)=t$ and then, for any $k > j$, $p_{k+1}(u,t)=p_0(t,t)=t$.
    It follows that all but finitely many members of the sequence belong to $T$. For a sufficiently large $i$ we thus have $p_i(u,t) \in T$ for every $u \in U$, $t \in T$, as required. 
\end{proof}

\subsection{Polynomial upper bound} \label{subsec:suff}

We prove here that $\alg A$ is in $\Ksurjeff$ whenever $\alg A$ has no nontrivial strongly abelian congruence.
The proof is done by studying ``tractable pieces'' of $\alg A$. The exact definition is somewhat convoluted; the following example is intended to provide some intuition behind the concept and the proof. 

\begin{example}
  Let $\alg B$ be the algebra with universe $B = \{0,1,2\}$ that consists of a single binary operation defined as follows.
$$
\begin{array}{r|ccc}
\cdot^{\alg B} & 0 & 1 & 2\\ 
\hline
0 & 0 & 0 & 2\\
1 & 0 & 1 & 1 \\
2 & 2 & 1 & 2
\end{array}
$$
This algebra is sometimes called the \emph{rock-paper-scissors} algebra, since it outputs the winner in that game (where, e.g., rock=0, scissors=1, paper=2)
Let $\alg A = \alg B^*$.
We sketch the reason why $\alg A \in \Kpolyeff$. 

Consider a homomorphism $h: \alg X \to \alg A$. Let $e$ be the term $e = x_1 \cdot 1$ and note that $e^{\alg A}(0)=0$ and $e^{\alg A}(1)=e^{\alg A}(2) = 1$. Since $h$ preserves all term operations, we in particular have $h(e^{\alg X}(x)) = e^{\alg{A}}(h(x))$ for every $x \in X$. This has two consequences.

First, it follows that $h$ maps $Y =  e^{\alg X}(X)$ into $e^{\alg{A}}(A) = \{0,1\}$. Since $\{0,1\}$ together with the appropriate restriction of $\cdot^{\alg A}$ is a semilattice, and homomorphisms to semilattices can be efficiently enumerated by \Cref{prop:semilattice}, it can be deduced that we can efficiently enumerate a set of functions $Y \to \{0,1\}$ that contains all the restrictions of homomorphisms $\alg{X} \to \alg{A}$ to the set $Y$. A more general version of this fact is item~(\ref{itm:Ks-tractable-piece}) in \Cref{lem:eff-and-trac}.
The second consequence of $e^{\alg{A}}(h(x)) = h(e^{\alg X}(x))$ is that $e^{\alg{A}} \circ h$ is determined by $h|_Y$, i.e., $h|_Y$ determines $h$ ``modulo'' the kernel of $e^{\alg{A}}$, which is the equivalence $\alpha = \{1,2\}^2 \cup \{0\}^2$. We can thus efficiently compute candidate homomorphisms modulo $\alpha$. A generalization of this fact is item~(\ref{itm:tractable-piece-kernel}) in \Cref{lem:properties_of_trac}. 

The proof now can be finished by similarly computing candidate homomorphisms modulo the equivalence $\beta = \{0,1\}^2 \cup \{2\}^2$,  combing the information modulo $\alpha$ and $\beta$ to get a list of candidate homomorphisms modulo $\alpha \cap \beta = 0_A$ (cf. item~(\ref{itm:trac-intersection-congruence}) in \Cref{lem:properties_of_trac}), and removing non-homomorphisms from the list (cf.
 item~(\ref{itm:tractable-piece-Ks}) in \Cref{lem:eff-and-trac}).
Another option to finish the proof is to use only the candidates modulo $\alpha$ and, for each such candidate $g$, to create a list containing all possible $f$ corresponding to $g$, cf. item~(\ref{itm:trac-from-block}) in \Cref{lem:properties_of_trac}.
\end{example}

A \emph{piece} of $\alg A$ is a pair $[P,\mu]$, where $P \subseteq A$ and $\mu$ is an equivalence relation on $P$.

\begin{definition} \label{def:tractable-piece}
    A piece $[P,\mu]$ of $\alg A$ is \emph{tractable} if there exists a finite-signature reduct $\alg B$ of $\alg A$ and a polynomial-time algorithm that, given an input algebra $\alg X$ in the signature of $\alg B$ and subset $Y \subseteq X$, outputs a set of mappings $Y \to P/\mu$ containing the following set (where $q_\mu: P \to P/\mu$ denotes the quotient mapping).
    $$
    \{q_\mu \circ h|_Y \mid h \colon \alg X \to \alg B \mbox{ is a homomorphism}, \ h(Y) \subseteq P \}
    $$
    We say that such a $\alg B$ \emph{witnesses} the tractability of $[P,\mu]$.
    We set 
    $$
    \trac = \{ [P,\mu] \mid [P, \mu] \mbox{ is a tractable piece of $\alg{A}$}\}.
    $$
\end{definition}

We also extend the definition of induced algebras from subsets to pieces: 
for a piece $[P,\mu]$ of $\alg A$, we define 
$\alg A|_{[P,\mu]}$ to be the algebra with universe $P$ and whose operations are all the operations of the form $f|_{P}$, where $f \in \Clo(\AddC{\alg A})$ preserves $P$ and $f|_P$ preserves $\mu$.

In the following, recall that $0_P$ is used to denote the equality relation on $P$. We sometimes disregard the formal difference between $P$ and $P/0_P$. 

The relation of tractable pieces to $\Kpolyeff$ is as follows.

\begin{lemma} \label{lem:eff-and-trac}
Let $\alg A$ be a finite algebra.
\begin{enumerate}
    \item\label{itm:tractable-piece-Ks} If $[A,0_A] \in \trac$, then $\alg A \in \Kpolyeff$.
    \item\label{itm:Ks-tractable-piece} If $(\alg A|_{[P,\mu]})/\mu \in \Kpolyeff$, then $[P,\mu] \in \trac$.
\end{enumerate}    
\end{lemma}
\begin{proof}
    (\ref{itm:tractable-piece-Ks}) 
    Let $\alg B$ be a finite-signature reduct of $\alg A$ witnessing that $[A,0_A]\in\trac$.
    Therefore, there is a polynomial-time algorithm that, given an input algebra $\alg X$ in the signature of $\alg B$, enumerates a set of maps that contains all the homomorphisms $\alg X\to\alg B$.
    In order to obtain an algorithm witnessing that $\alg A\in\Kpolyeff$, it suffices to filter out from this set the maps that are not homomorphisms.

    (\ref{itm:Ks-tractable-piece}) If $(\alg A|_{[P,\mu]})/\mu \in \Kpolyeff$, then there exists a finite-signature reduct $\alg B$ of $\alg A|_{[P,\mu]}$ having $\alg P$ as a subalgebra, itself having $\mu$ as a congruence, and such that $\alg P/\mu$ is in $\Kpolyeff$.
    Consider a homomorphism $h$ from an algebra $\alg X$ to $\alg B$ such that $h(Y) \subseteq P$, and let $\alg Z$ be the subalgebra of $\alg X$ generated by $Y$. 
    Since the preimage of a subuniverse under a homomorphism is a subuniverse, the restriction $h|_Z$ maps $Z$ into $P$, so $h|_Z$ is  a homomorphism $\alg Z \to \alg P$ and $q_{\mu} \circ h|_Z$ a homomorphism $\alg Z \to \alg P/{\mu}$.
    Therefore, in order to obtain an algorithm witnessing that $[P,\mu]\in\trac$, it suffices to do the following on an input $\alg X$, $Y\subseteq X$: compute the subalgebra $\alg Z$ generated by $Y$ in $\alg X$, enumerate all the homomorphisms $\alg Z\to\alg P/{\mu}$, which is possible since $\alg P/{\mu}\in\Kpolyeff$, and restrict the obtained mappings to the set $Y$.
\end{proof}

The next lemma will allow us to produce new tractable pieces from those already known.

\begin{lemma} \label{lem:properties_of_trac}
Let $\alg A$ be a finite algebra.
\begin{enumerate}
    \item\label{itm:trac-subset} If $[P,\mu] \in \trac$ and $Q \subseteq P$, then $[Q, \mu \cap (Q \times Q)] \in \trac$
    \item\label{itm:trac-intersection-congruence} If $[P,\mu], [P,\nu] \in \trac$, then $[P, \mu \cap \nu] \in \trac$.
    \item\label{itm:trac-from-block} If $[P,\mu] \in \trac$ and $[Q,0_Q] \in \trac$, where $Q$ is a $\mu$-equivalence class, then $[P,\mu \cap \nu] \in \trac$, where $\nu = 0_Q \cup (P \setminus Q)^2$. 
    \item\label{itm:tractable-piece-kernel} If $P,Q \subseteq A$, $e \in \Clo_1(\alg A)$ is such that $e(P) \subseteq Q$, and $[Q,0_Q] \in \trac$, then $[P, \ker e|_{P}] \in \trac$.
\end{enumerate}
\end{lemma}

\begin{proof}
  Let $\alg B$ be a finite-signature reduct of $\alg A$ witnessing the tractability of every piece in $\trac$; such a $\alg B$ exists by taking an algebra having all the operations from finite-signature reducts witnessing the membership of the finitely many pieces $[P,\mu]\in\trac$.
  
  (\ref{itm:trac-subset}) Let $\alg X$ be an algebra in the signature of $\alg B$ and $Y\subseteq X$. Let $\nu$ be $\mu\cap (Q\times Q)$.
  In order to enumerate a set containing all the maps of the form $q_\nu\circ h|_Y$, where $h$ is a homomorphism $\alg X\to\alg B$ such that $h(Y)\subseteq Q$, it suffices to enumerate a set containing all the  maps $q_\mu\circ h|_Y$ with $h(Y)\subseteq P$, discard those that do not satisfy $h(Y)\subseteq Q$, and restrict the remaining to the set $Q$.

  (\ref{itm:trac-intersection-congruence})
  Let $\alg X$ be an algebra in the signature of $\alg B$ and $Y\subseteq X$.
  Let $S_\mu$ be the set of maps enumerated by an algorithm witnessing that $[P,\mu]\in\trac$, and similarly for $S_\nu$.
  Given $f\in S_\mu$ and $g\in S_\nu$, one can check whether for every $y\in Y$, there exists an $(\mu\cap \nu)$-equivalence class $C_y$ contained in $f(y)$ and $g(y)$.
  If it is the case, then we add the mapping $y\mapsto C_y$ to the enumeration.
  In particular, for every $h\colon \alg X\to \alg B$ such that $h(Y)\subseteq P$, the maps $q_\mu\circ h|_Y$ and $q_\nu\circ h|_Y$ are in $S_\mu$ and $S_\nu$, respectively, and they satisfy the property above, so that $q_{\mu\cap \nu}\circ h|_Y$ is enumerated by the algorithm.
  Since there is a polynomial number of maps in $S_\mu$ and $S_\nu$, this algorithm runs in polynomial time.

  (\ref{itm:trac-from-block})
    Let $\alg X$ be an algebra in the signature of $\alg B$ and $Y\subseteq X$.
    We describe an algorithm witnessing that $[P,\mu\cap\nu]\in\trac$.
    Let $S$ be the set of mappings enumerated by an algorithm witnessing the fact that $[P,\mu]\in\trac$, when applied to the input $\alg X$ and $Y$.
    Consider a mapping $g\in S$. Let $Y'=g^{-1}(\{Q\})$ and
    let $S_g$ be the set of mappings enumerated by an algorithm witnessing the fact that $[Q,0_Q]\in\trac$, when applied on the input $\alg X$ and $Y'$.
    For each $f\in S_g$, add to the enumeration the mapping $\tilde f\colon Y\to P/(\mu\cap\nu)$ defined by $\tilde f(y)=g(y)$ if $y\in Y\setminus Y'$, and $\tilde f(y)=\{f(y)\}$ if $y\in Y'$.
    It remains to check that every $q_{\mu\cap\nu}\circ h|_Y$ is enumerated by this algorithm, where $h\colon \alg X\to\alg B$ is a homomorphism such that $h(Y)\subseteq P$.
    We have that $g=q_{\mu}\circ h|_Y$ is in $S$. Let  $Y'=g^{-1}(\{Q\})$ and note that we also have that $f=h|_{Y'}$ is in $S_{g}$.
    Clearly, $q_{\mu\cap\nu}\circ h$ is equal to $\tilde f$ and is thus enumerated by the algorithm.

  (\ref{itm:tractable-piece-kernel})
  Let $\alg X$ be an algebra in the signature of $\alg B$ and $Y\subseteq X$.
  Let $S$ be the set of mappings enumerated by an algorithm witnessing the fact that $[Q,0_Q]\in\trac$, when applied on the input $\alg X$ and $Z=e^{\alg X}(Y)$.
   For every $g\in S$, add to the enumeration the mapping $\tilde g\colon Y\to P/{\ker e|_P}$ that sends $y$ to the $\ker e|_P$-equivalence class ${(e^{\alg A}|_P)}^{-1}(g(e^{\alg X}(y)))$.
   We show that every mapping $q_{\ker e|_P}\circ h$ is enumerated, where $h\colon\alg X\to\alg B$ is such that $h(Y)\subseteq P$.
   Note that since $h$ is a homomorphism, we have $h(e^{\alg X}(y))=e^{\alg A}(h(y))$, for every $y\in Y$.
   Therefore, $h(Z)\subseteq e^{\alg A}(P) \subseteq Q$.
   In particular, $g=h|_Z$ is in $S$.
   It remains to observe that $q_{\ker e|_P}\circ h|_Y$ is the same as $\tilde g$, and therefore it is enumerated by the algorithm.
\end{proof}

At this point, \Cref{def:tractable-piece} can be forgotten; we will only work with~\Cref{lem:eff-and-trac,lem:properties_of_trac}. 
The proof is finished in two steps. The following corollary of~\Cref{lem:properties_of_trac} is useful for both of them.

\begin{corollary}\label{cor:induction}
 Let $\alg A$ be a finite algebra and $[P,\mu]$ a piece of $\alg A$.
 \begin{enumerate}
     \item\label{itm:induction-separation} If for every  $a,b \in P$ with $a\neq b$  there is a unary term operation $e^{ab} \in \Clo_1(\alg A)$ such that $e^{ab}(a) \neq e^{ab}(b)$ and $[e^{ab}(P),0_{e^{ab}(P)}] \in \trac$, then $[P,0_P] \in \trac$.
     \item\label{itm:induction-classes} If $[P,\mu] \in \trac$ and $[Q,0_Q] \in \trac$ for every equivalence class $Q$ of $\mu$, then $[P,0_P] \in \trac$.
 \end{enumerate} 
\end{corollary}

\begin{proof}
    We first prove (\ref{itm:induction-separation}). By item (\ref{itm:tractable-piece-kernel}) of~\Cref{lem:properties_of_trac}, we have that $[P,\ker e^{ab}|_P]\in\trac$ for every $a,b\in P$ with $a\neq b$.
    Since the equivalence relation $\bigcap_{a\neq b} \ker e^{ab}$ is ${0_P}$, we get by item (\ref{itm:trac-intersection-congruence}) of~\Cref{lem:properties_of_trac} that $[P,0_P]\in\trac$.
    
    Concerning (\ref{itm:induction-classes}),
    it suffices to apply item (\ref{itm:trac-intersection-congruence}) of~\Cref{lem:properties_of_trac} to all the pieces $[P,\mu\cap ({0_Q}\cup (P\setminus Q)^2)]$, where $Q$ is an equivalence class of $\mu$. The fact that all these pieces are in $\trac$ comes from item (\ref{itm:trac-from-block}) of ~\Cref{lem:properties_of_trac}.
\end{proof}

The first step is to prove that $\alg A \in \Ksurjeff$ whenever $\alg A$ has a chain of congruences of type different than 1.
 A special case stated in the following lemma was already essentially done in~\Cref{subsec:examples}.

\begin{lemma}\label{lem:minimal-trac}
Every finite minimal algebra that is not of type $1$ is in $\Ksurjeff$.
\end{lemma}

\begin{proof}
Let $\alg M$ be a finite minimal algebra  whose type is not $1$.
By~\Cref{thm:palfy}, there exists an algebra $\alg A$ such that $\Clo(\alg A)\subseteq \Clo(\AddC{\alg M\,})$ and such that $\alg A$ is either a two-element semilattice or a group.
Thus, by~\Cref{prop:semilattice,prop:groups}, $\alg A$ is in~\Ksurjeff.
    By~\Cref{prop:clone-Kpoly}, $\AddC{\alg M\,}$ is also in \Ksurjeff.
    Finally, by~\Cref{prop:Ksurj-polynomials}, $\alg M$ is in \Ksurjeff.
\end{proof}

\begin{theorem}\label{thm:omit-type-1}
    Let $\alg A$ be a finite algebra with a sequence $\alpha_0,\dots,\alpha_n$ of congruences  such that
    \begin{itemize}
        \item $\alpha_0=0_A$, $\alpha_n=1_A$, and
        \item for all $i<n$, 
        $(\alpha_i,\alpha_{i+1})$ is a cover
        and
        $\typ(\alpha_i,\alpha_{i+1})\neq 1$.
    \end{itemize}
    Then $\alg A$ is in $\Ksurjeff$.
\end{theorem}
\begin{proof}
We prove that $\AddC{\alg A}$ is in $\Kpolyeff$, which is enough by~\Cref{prop:Ksurj-polynomials} and since every homomorphism to $\AddC{\alg A}$ must be surjective.
Recall that every congruence of $\alg A$ is also a congruence of $\AddC{\alg A}$.
The proof is by induction on $n$.
    The case $n=0$ is clear, since then $|A|=1$.
    
    Suppose $n>0$. The quotient algebra $\AddC{\alg A}/\alpha_1$ has the sequence of congruences $0_{A/\alpha_1}=\alpha_1/\alpha_1\subsetneq\alpha_2/\alpha_1\subsetneq\dots\subsetneq\alpha_n/\alpha_1=1_{A/\alpha_1}$.
    By \Cref{thm:type-quotient}, $\typ_{\AddC{\alg A}/\alpha_1}(\alpha_i/{\alpha_1},\alpha_{i+1}/\alpha_1) = \typ_{\AddC{\alg A}}(\alpha_i,\alpha_{i+1})$, therefore the induction hypothesis gives us $\AddC{\alg A}/\alpha_1 \in \Kpolyeff$.
    It then follows from item (\ref{itm:Ks-tractable-piece}) in~\Cref{lem:eff-and-trac} that $[A,\alpha_1] \in \trac[\AddC{\alg A}]$. By item (\ref{itm:induction-classes}) in~\Cref{cor:induction} and item (\ref{itm:tractable-piece-Ks}) in~\Cref{lem:eff-and-trac}, it is now enough to verify that $[P,0_P]$ is in $\trac[\AddC{\alg A}]$ for every equivalence class $P$ of $\alpha_1$.

    Let $P$ be such an equivalence class, let $U$ be a $(0_A,\alpha)$-minimal set, and consider any distinct $a,b \in P$. By item (\ref{itm:separation}) in~\Cref{thm:HM-separation}, there exists $e^{ab} \in \Clo_1(\AddC{\alg A})$ such that $e^{ab}(a) \neq e^{ab}(b)$ and $e^{ab}(A) = U$. The image $e^{ab}(P)$ is contained in $U$ and in an equivalence class of $\alpha_1$ (as $e^{ab}$ preserves $\alpha_1)$. Since $e^{ab}(P)$ also contains two distinct elements, namely $e^{ab}(a)$ and $e^{ab}(b)$, it is contained in a trace $N$. Since
    $\AddC{\alg A}|_N$ is a minimal algebra of type different than $1$,
    by~\Cref{lem:minimal-trac}, $\AddC{\alg A}|_N = \AddC{\alg A}_{[N,0_N]} \in \Kpolyeff$, which in turn implies $[N,0_N] \in \trac[\AddC{\alg A}]$ by item (\ref{itm:Ks-tractable-piece}) of \Cref{lem:eff-and-trac}, and then $[e^{ab}(P),0_{e^{ab}(P)}] \in \trac[\AddC{\alg A}]$ by item (\ref{itm:trac-subset}) of \Cref{lem:properties_of_trac}. 
    An application of item  (\ref{itm:induction-separation}) in~\Cref{cor:induction} gives that $[P,0_P]$ is in $\trac[\AddC{\alg A}]$, which finishes the proof of the theorem.
\end{proof}

The theorem already covers a large class of algebras, e.g., all so-called Taylor algebras (see, e.g., \cite{DBLP:conf/lics/BartoDM21}  for a discussion about this class and another proof of the fact that finite Taylor algebras are in $\Ksurjeff$).
However, not all algebras without strongly abelian minimal congruences are covered, e.g., the algebra in~\Cref{ex:not-closed-quotient}. For the second step, we start almost from scratch: we only use~\Cref{thm:omit-type-1} to derive the following lemma. 

\begin{lemma} \label{lem:tctstuff}
Let $\alpha$ be a minimal congruence of $\alg A$, $U$ be a $(0_A,\alpha)$-minimal set, and suppose that $\typ(0_A,\alpha)\neq 1$.
Denote $B$ the body of $U$ and $T$ the tail (so $U$ is a disjoint union of $B$ and $T$). Then $[U,\mu] \in \trac[\AddC{\alg A}]$ where $\mu = {0_B} \cup T^2$. 
\end{lemma}

\begin{proof}
Let $\nu$ be the equivalence on $U$ whose classes are $B$ and $T$.
By item~(\ref{itm:body-is-nice}) in \Cref{thm:pseudo-malcev-meet}, $\AddC{\alg A}|_{[B,0_B]} = \AddC{\alg A}|_{B}$ has a Mal'cev term operation or a binary commutative idempotent term operation.
Therefore, $\AddC{\alg A}|_B$ has no congruence covers of type 1 by \Cref{thm:taylor-and-types} and is then in $\Ksurjeff$ by ~\Cref{thm:omit-type-1} applied to any sequence of covers from $0_A$ to $1_A$. 
This implies that $[B,0_B]\in\trac[\AddC{\alg A}]$ by item (\ref{itm:Ks-tractable-piece}) in~\Cref{lem:eff-and-trac}.

We also know by item~(\ref{itm:body-and-tail}) in~\Cref{thm:pseudo-malcev-meet} that 
$\AddC{\alg A}|_{[U,\nu]}/\nu$ contains a binary commutative idempotent operation so $[U,\nu] \in \trac[\AddC{\alg A}]$ by the same argument.
Now item (\ref{itm:trac-from-block}) in~\Cref{lem:properties_of_trac} gives us the conclusion. 
\end{proof}

We are now ready to prove the general membership result. 

\begin{theorem}\label{thm:membership}
Let $\alg A$ be a finite algebra without nontrivial strongly abelian congruences. Then $\alg A \in \Ksurjeff$. 
\end{theorem}

\begin{proof}
We will prove the following two claims by induction on $n$.

\begin{description}
\item[$(I_{n})$] For every minimal congruence $\alpha$ of $\alg A$,
for every $(0_A,\alpha)$-minimal set $U$ with tail $T$, and every $P \subseteq U$ such that $|P \cap T| \leq n$, we have $[P,0_P] \in \trac[\AddC{\alg A}]$.
\item[$(II_n)$] For every $P \subseteq A$ with $|P| \leq n$, we have $[P,0_P] \in \trac[\AddC{\alg A}]$.
\end{description}
Note that $(II_{|A|})$ implies that $\AddC{\alg A}\in\Kpolyeff$ by~\Cref{lem:eff-and-trac}. Then $\AddC{\alg A}\in\Ksurjeff$ since every homomorphism to $\AddC{\alg A}$ is surjective, which implies that $\alg A\in\Ksurjeff$ by~\Cref{prop:Ksurj-polynomials}.

The induction starts by observing that $(I_1)$ holds by~\Cref{thm:strongly-abelian-omit-one}, \Cref{lem:tctstuff}, and item (\ref{itm:trac-subset}) of \Cref{lem:properties_of_trac}.

We now prove that if $(I_{n-1})$ holds, then $(II_n)$ holds. %
  Consider any $P \subseteq A$ with $|P| \leq n$.
  We claim that for every $a \neq b$ in $P$, there exists $f^{ab} \in \Clo_1(\AddC{\alg A})$ such that $[f^{ab}(P),0_{f^{ab}(P)}] \in \trac[\AddC{\alg A}]$ and $f^{ab}(a) \neq f^{ab}(b)$.
  From (\ref{itm:induction-separation}) in~\Cref{cor:induction} it then follows that $[P,0_P] \in \trac[\AddC{\alg A}]$, as required.
  
  So, let $a,b \in P$ be distinct. Let $\alpha$ be a minimal congruence contained in the congruence generated by $(a,b)$, let $U$ be a $(0_A, \alpha)$-minimal set, let $c,d$ be distinct elements in a trace of $U$, and let $e \in \Clo_1(\AddC{\alg A})$ be such that $e|_{U} = \operatorname{id}_U$  and $e(A) = U$ (given by \Cref{thm:HM-separation}).
  
  By the choice of $\alpha$, $c$, and $d$, the pair $(c,d)$ is in the congruence generated by $\{(a,b)\}$. Therefore, because of the way how congruences are generated, there exists a sequence of elements
  $$
  c= c_0, \ c_1, \ \dots, \ c_k = d, \quad c_i \in A
  $$
  and a sequence of term operations 
  $$
  f_1, \ f_2, \ \dots, \ f_k, \quad f_i \in \Clo_1(\AddC{\alg A}), \ f_i(\{a,b\}) = \{c_{i-1},c_i\}.
  $$
  By replacing $f_i$ with $e \circ f_i$ and $c_i$ with $e(c_i)$ we can further assume that $f_i(P) \subseteq U$ and all the $c_i$ are in $U$ (note that $c_0$ and $c_k$ do not change since $e|_U$ is the identity on $U$ and $c,d \in U$). As $c \neq d$, there exists $i$ such that $c_{i-1}=c$ and $c_{i-1} \neq c_i$. But now $f_i$ can be taken for $f^{ab}$: the set $f_i(P)$ is contained in $U$, $f_i(a) \neq f_i(b)$, and  $f_i(P)$ has at most $n$ elements and contains an element in the body of $U$, so that $|f_i(P)\cap T|\leq n-1$.
  By $(I_{n-1})$, we get $[f_i(P),0_{f_i(P)}] \in \trac[\AddC{\alg A}]$.     

  We now prove that if $(II_n)$ holds, then $(I_n)$ holds. %
  Consider $P$ as in the statement of $(I_n)$. By~\Cref{lem:tctstuff}  we have that $[P,\mu] \in \trac[\AddC{\alg A}]$, where $\mu$ is the equivalence whose only non-singleton equivalence class is $P \cap T$. By item (\ref{itm:trac-from-block}) of~\Cref{lem:properties_of_trac}, it is enough to show that $[P \cap T, 0_{P \cap T}] \in \trac[\AddC{\alg A}]$. Since $|P\cap T|\leq n$, this is guaranteed by $(II_n)$.
\end{proof}

The proof presented in this section is not very long. This was made possible, apart from the tame congruence theory, by the somewhat unnatural concept of a tractable piece, with which we could work axiomatically after we established  its properties. The two main unsatisfactory aspects are the ad hoc choice of properties of $(0,\alpha)$-minimal sets in \Cref{thm:pseudo-malcev-meet} and the two step process, in which we ``almost proved'' the result in~\Cref{thm:omit-type-1} and then started again almost from scratch. It could be worth the effort to keep looking for a more natural and straightforward proof.

\section{Conclusion} \label{sec:summary}

Two refinements of~\Cref{thm:main-official} follow immediately from our results.

\begin{corollary}\label{cor:main-surj}
    Let $\alg A$ be a finite algebra. The following are equivalent:
    \begin{enumerate}
        \item $\alg A$ is in $\Ksurjeff$,
        \item $\alg A$ is in $\Ksurj$, that is,  $\counting{\alg A}^s(n) \in O(n^k)$ for some integer $k$,
        \item $\counting{\alg A}^s(n) \in 2^{o(n^{1/k})}$, where $k=\lfloor\log_2|A|\rfloor$,
        \item\label{itm:no-minimal-congruence} No nontrivial congruence $\alpha$ of $\alg A$ is strongly abelian.
    \end{enumerate}
\end{corollary}
\begin{proof}
    (1) clearly implies (2), and (2) clearly implies (3).
The implication from (3) to (4) follows from~\Cref{cor:nonmembership}.
The implication from (4) to (1) follows from~\Cref{thm:membership}.
\end{proof}

\begin{corollary}\label{cor:main}
    Let $\alg A$ be a finite algebra. The following are equivalent:
    \begin{enumerate}
        \item $\alg A$ is in $\Kpolyeff$,
        \item $\alg A$ is in $\Kpoly$, that is,  $\counting{\alg A}(n) \in O(n^k)$ for some integer $k$,
        \item \label{itm:exp-bound} $\counting{\alg A}(n) \in 2^{o(n^{1/k})}$, where $k=\lfloor\log_2|A|\rfloor$,
        \item No subalgebra of $\alg A$ has a nontrivial strongly abelian congruence.
    \end{enumerate}
\end{corollary}
\begin{proof}
 The only interesting implications are from (3) to (4) and (4) to (1).

 (4) implies (1). Suppose that no subalgebra of $\alg A$ has a nontrivial strongly abelian congruence. Then by~\Cref{cor:main-surj}, all the subalgebras of $\alg A$ are in $\Ksurjeff$.
 By~\Cref{prop:equivalences-Kpoly} and the remark following it, it follows that $\alg A$ is in $\Kpolyeff$.

 (3) implies (4). We prove this implication by contraposition. Suppose that some subalgebra $\alg B$ of $\alg A$ has a nontrivial strongly abelian congruence.
 Let $\ell=\lfloor\log_2|B|\rfloor$.
 By~\Cref{cor:main-surj}, we get that $\counting{\alg B}^s(n) \not\in 2^{o(n^{1/\ell})}$, so in particular $\counting{\alg B}^s(n)\not\in 2^{o(n^{1/k})}$.
 Since $\counting{\alg A}(n)\geq\counting{\alg B}^s(n)$, we obtain the result.
\end{proof}

Recall that the equivalent conditions of the above corollaries can be verified in polynomial time for finite algebras of finite signatures by \Cref{prop:efficient-decision} and that the bound in item~(\ref{itm:exp-bound}) is essentially tight by 
\Cref{cor:matrix-powers}.

We remark that growth rates of some counting sequences are interesting also beyond finite $\alg A$ and finite $n$; indeed, the classification theory~\cite{Shelah-classification} provides deep results under the assumption that the spectrum (which is the same as G-spectrum for small signatures on infinite cardinal $\kappa$) is strictly smaller than the maximum one on an infinite $\kappa$. However, for our counting sequence $\counting{\alg A}(\kappa)$, the infinite versions do not seem fruitful. Note that if $\kappa$ is infinite and the signature is at most countable, then $\counting{\alg A}(\kappa) = |A|^{\kappa}$ since even the absolutely free algebra on $\kappa$ generators has cardinality $\kappa$ in this case, giving us the lower bound $\counting{\alg A}(\kappa) \geq |A|^{\kappa}$ which matches the trivial upper bound. Characterizing $\counting{\alg{A}}(n) \in O(n^k)$ may be nontrivial for finite $n$ and infinite $\alg A$ but it has a different flavor because the $k$ in our result is related to $|A|$.

We now turn to the complexity-theoretic application,
\Cref{thm:csp-corollary}, which also follows from the obtained results in a straightforward manner. We restate the theorem for convenience.

\cspcorollary*
\begin{proof} 
    The implication from (1) to (2) follows from the assumption that P is different from NP, while the implication from (3) to (1) is essentially by~\Cref{cor:main}.
    Indeed, any homomorphism $\str X\to\str B$, if one exists, must be in particular a homomorphism from the algebraic reduct $\alg X$ corresponding to the signature of $\alg A$ to $\alg A$.
    By the equivalence of items (1) and (3) in~\Cref{cor:main}, one can enumerate all such homomorphisms, and check if one is a homomorphism $\str X\to\str B$.

    We prove the implication from (2) to (3) by contraposition.
    Suppose that $\counting{\alg A}$ is not polynomially bounded.
    By the equivalence of items (2) and (4) in~\Cref{cor:main},
    there exist a subalgebra $\alg B$ of $\alg A$ and a nontrivial strongly abelian congruence $\alpha$ of $\alg B$.
    Let $a\neq b$ be elements of $\alg B$ be such that $(a,b)\in\alpha$.
    Let $\str B$ be the expansion of $\str A$ by the ternary relation $R^{\str B}$ containing the tuples $(a,a,b),(a,b,a),(b,a,a)$.
    Let $\str C$ be the structure with universe $\{a,b\}$ in the signature containing only $R$ interpreted as $R^{\str C} = R^{\str B}$. 
    Then $\csp(\str C)$ is known as the 1-in-3-SAT problem and is NP-complete (see, e.g., \cite{Schaefer:1978}).

    We prove that $\csp(\str B)$ is NP-complete by constructing a polynomial-time reduction from $\csp(\str C)$ to $\csp(\str B)$. 
    Let $\str X$ be an instance of $\csp(\str C)$, whose domain has size $|X|=n \geq 1$. By removing superfluous elements, we can assume that every element $x\in X$ belongs to a tuple in $R^{\str X}$.
    Let $\alg F = \alg F_{\alg A,ab}(n)$ be the ab-free algebra constructed in~\Cref{prop:abfree}.
    By~\Cref{prop:abfree-are-small}, $F$ has size $O(n^k)$, where $k$ is a constant, and it can be computed in polynomial time. We rename the elements in the universe so that $X \subseteq F$ and $X$ is the $ab$-free set of $\alg F$. 
    Let $\str F$ be the expansion of $\alg F$ to the signature of $\str B$ defined by $R^{\str F}=R^{\str X}$ and $S^{\str F}=\emptyset$ for all the remaining relation symbols $S$. 
    
    We need to show that there exists a homomorphism $\str X\to \str C$ if, and only if, there exists a homomorphism $\str F\to \str B$.
    Suppose that there exists a homomorphism $h\colon\str X\to \str C$.
    Since $\alg F$ is $ab$-free, there exists a homomorphism $h' \colon \alg F \to \alg A$ extending $h$. By construction, $h'$ preserves $R$ as well as all the other relations, so it is a homomorphism $\str F \to \str B$.
    Conversely, suppose that there exists a homomorphism $h\colon \str F\to \str B$. Since every $x \in X$ belongs to a tuple in $R^{\str X}$ and $R^{\str B}$ only contain values in $\{a,b\}$, $h$ maps $X$ into $\{a,b\}$. The restriction of $h$ to $X$ is then a homomorphism $\str X\to \str C$.
    \end{proof}

As we have already mentioned, the complexity classification of CSPs over finite structures remains open even for finite algebras. Note that $\csp(\alg A)$ can be trivially solvable even if $c_{\alg A}$ grows exponentially. For instance, if $\alg A$ contains an element $a$ such that $f^{\alg A}(a,a, \dots, a)=a$ for every $f$ in the signature, then the constant mapping with image $\{a\}$ is always a homomorphism to $\alg A$, so deciding $\csp(\alg A)$ is trivial; such algebras $\alg A$ can clearly have exponential growth. Interestingly, the investigation of computational complexity of CSPs over finite relational structures can be reduced to investigating finite \emph{idempotent} algebras, i.e., those where each element $a$ has the above property. This is among the reasons why the new universal algebraic theories that emerged in that context (see \cite{minimal-taylor} for comparison) have small intersection with tame congruence theory. On the other hand, our partial result toward classifying the complexity of CSPs over finite algebras uses tame congruence theory and does not use the newly emerged ones at all. The project of classifying the complexity of CSPs over finite algebras (or general finite structures) could therefore also force one to unify these theories, which would be a much desired outcome. 

For finite algebras $\alg A$ in finite signature with $\counting{\alg A}(n) \in O(n^k)$, our results yield a polynomial algorithm algorithm for $\csp(\alg A)$, which is however nonuniform in that the running time depends on $\alg A$. \emph{Is there a uniform polynomial-time algorithm in this case?} That is, is there a polynomial-time algorithm that, given finite $\alg X$ and $\alg A$ such that $\counting{\alg A}(n) \in O(n^k)$ for some $k$, decides whether $\alg X$ has a homomorphism to $\alg A$? 
Such an algorithm cannot be based on simply providing uniform bound on $k$, since there is no such, as witnessed by powers of the 2-element group.

There are many purely mathematical questions concerning the growth rates of $\counting{\alg A}$ for finite $\alg A$. Here are some. \emph{Is it possible to (almost) exactly compute $\counting{\alg A}$ for interesting classes of algebras?}
 \emph{Is it possible to characterize sequences of the form $\counting{\alg A}$ (where $\alg A$ ranges through all finite algebras or  algebras from a specific class)?}
 \emph{Are there nontrivial lower bounds on $\counting{\alg{A}}(n)$ (other than \Cref{cor:nonmembership})?}
 \emph{When is $\counting{\alg A}$ upper bounded by a linear (quadratic, \dots, or sublinear, logarithmic, \dots) function?}
 Observe that $\counting{\alg A}$ is always at most exponential, and at least logarithmic as witnessed by direct powers of $\alg A$. It can be logarithmic, e.g., if $\alg A$ is the 2-element Boolean algebra. 
 Similar questions can be interesting for the sequence counting the minimal size of a generating set from \Cref{q:generating}.
 We also remark that for the other counting sequences mentioned in \Cref{subsec:related-counting}, there are interesting results in this spirit, e.g.,
 a finite group is nilpotent if, and only if, its free spectrum is in $2^{O(n^k)}$ \cite{Neumann63,Higman67} and this happens if, and only if, its G-spectrum is \cite{generative-complexity,IM01}. Several results of this sort are also provided in \cite{Berman:2010} for counting sequences closely related to generating sets of subalgebras of powers, too.

\ifarxiv
\bibliographystyle{plainnat}
\fi
\bibliography{refs}
\end{document}